\documentclass[12pt,twoside]{preprint}
\usepackage{filecontents}
\RequirePackage{amsthm,amsmath}
\RequirePackage[numbers]{natbib}
\RequirePackage[colorlinks,citecolor=blue,urlcolor=blue]{hyperref}
\usepackage[all,arc]{xy}
\usepackage{mathrsfs }
\usepackage{latexsym,amssymb,amsmath,amsfonts,graphicx,color,fancyvrb,amsthm,enumerate,natbib}
\usepackage[usenames,dvipsnames]{xcolor}
\usepackage{mathtools}
\usepackage{float}
\usepackage{subfigure}
\usepackage{epstopdf}
\makeatletter
\newcommand*{\rom}[1]{\expandafter\@slowromancap\romannumeral #1@}
\makeatother
\usepackage{bbm}
\allowdisplaybreaks
\newtheorem{thm}{Theorem}[section]
\newtheorem{fact}[thm]{Fact}
\newtheorem{cor}[thm]{Corollary}

\newtheorem{lem}[thm]{Lemma}

\theoremstyle{definition}
\newtheorem{defn}[thm]{Definition}

\theoremstyle{remark}
\newtheorem{rem}[thm]{Remark}

\DeclareMathOperator*{\argmax}{arg\,max}

\begin{document}

	\title{Comparison of Bayesian and Frequentist Multiplicity Correction For Testing
		Mutually Exclusive Hypotheses Under Data Dependence }
	\author{Sean Chang$^1$, James O. Berger$^2$}
	\institute{Duke University, US, \email{sean.chang@duke.edu}
		\and{Duke University, US, \email{berger@stat.duke.edu}}}

	\maketitle

		\begin{abstract}
			The problem of testing mutually exclusive hypotheses with dependent test statistics is considered. Bayesian and frequentist approaches to
			multiplicity control are studied and compared to help gain understanding as to the effect of test statistic dependence on each approach.
			The Bayesian approach is shown to have excellent frequentist properties and is argued to be the most effective
			way of obtaining  frequentist multiplicity control, without sacrificing power, when there is considerable test statistic dependence.
		\end{abstract}

	\section{Introduction}
	
	Modern scientific experiments often require considering a large number of hypotheses simultaneously (\cite{efron2004large}, \cite{noble2009does}) and has led to
	extensive interest in controlling for multiple testing (henceforth, just termed controlling for multiplicity). Many multiplicity control methods have been proposed in the
	frequentist literature, such as the Bonferroni procedure which controls the family-wise error rates, and various versions of false discovery rates  (cf.$\text{ }$\cite{benjamini1995controlling}
	$\text{ }$ and 
	\cite{storey2003positive}) which control for the fraction of false discoveries to stated discoveries. The asymptotic behavior of false discovery rate has been studied in 
	\cite{abramovich2006adapting}.
	
	The Bayesian approach to controlling for multiplicity operates through the prior probabilities assigned to hypotheses. For instance,
	in the scenario that is considered herein of testing mutually exclusive hypotheses (only one of $n$ considered hypotheses can be true), one can simply
	assign each hypothesis prior probability equal to $1/n$ and carry out the Bayesian analysis; this automatically controls for
	multiplicity. That multiplicity is controlled through prior probabilities of  hypotheses or models is extensively discussed in
	\cite{scott2006exploration}, \cite{scott2010bayes}, \cite{berger2014bayesian} for a two-groups model, variable-selection in linear models, and subset analysis, respectively.
	
	One of the appeals of the Bayesian approach to multiplicity control is that it does not depend on the dependence structure of the test statistics; the Bayes procedure will automatically adapt to the dependence structure through Bayes theorem, but the prior probability assignment that is controlling for multiplicity is unaffected by dependence. In contrast, frequentist approaches to multiplicity control are usually highly affected by test statistic dependence. For instance, the Bonferroni correction is fine if the test statistics for the hypotheses being tested are independent, but can
	be much too conservative (losing detection power), if the test statistics are dependent.
	
	An interesting possibility for frequentist multiplicity control in dependence situations is thus to develop the procedure in a Bayesian fashion and verify that the procedure has sufficient control from a frequentist perspective. This has the potential of yielding optimally powered
	frequentist procedures for multiplicity control. There have been other papers that study the frequentist properties of Bayesian multiplicity
	control procedures (\cite{bogdan2008comparison}, \cite{guo2010multiplicity}, \cite{abramovich2006bayesian}), but they have not focused on the situation of data dependence.
	
	We investigate the potential for this program by an exhaustive analysis of the simplest multiple testing problem which exhibits test statistic dependence.
	The data $\boldsymbol{X}=(X_1, \ldots X_n)'$ arises from the multivariate normal distribution
	\begin{equation}
	\boldsymbol{X}\sim multinorm \left(
	\begin{pmatrix}
	\theta_1   \\
	\theta_2  \\
	\vdots  \\
	\theta_n
	\end{pmatrix}
	,
	\begin{pmatrix}
	1 & \rho & \cdots & \rho \\
	\rho & 1 & \cdots & \rho \\
	\vdots  & \vdots  & \ddots & \vdots  \\
	\rho & \rho & \cdots & 1
	\end{pmatrix}  \right) \,,
	\end{equation}
	where $\rho$ is the correlation between the observations. Consider testing the $n$ hypotheses
	$M_0^i: \theta_i=0$ versus $M_1^i: \theta_i \neq 0$, but under the assumption that at most
	one alternative hypothesis could be true. (It is possible that no alternative is true.) Although our study of this problem is
	pedagogical in nature, such testing problems can arise in signal detection,
	when a signal could arise in one and only one of $n$ channels, and there is common
	background noise in all channels, leading to the equal correlation structure. We will, for convenience in exposition,
	use this language in referring to the situation.
	
	In Section 2 we introduce two natural frequentist procedures for multiplicity control in this problem and, in Section 3,
	we introduce a natural Bayesian procedure. Section 4 explores a highly curious phenomenon that is encountered when $\rho$ is near 1; when $n >2$, the Bayesian procedure finds the true alternative
	hypothesis with near certainty, while an ad hoc frequentist procedure fails to do so.
	Sections 5 and 6 study the frequentist properties of the original Bayesian procedure and a Type-II MLE approach, showing that, as $n \rightarrow \infty$, the Bayesian procedures have strong frequentist
	control of error. 
	Section 7 considers the situation in which there
	is a data sample of growing size $m$ for each $\theta_i$.

	\section{Frequentist Multiplicity Control}
	Two natural frequentist procedures are considered.
	\subsection{An Ad hoc Procedure}
	Declare channel $i$ to have the signal if $\max \limits_{1\leq j \leq n} |X_j|>c$, where $c$ is determined to achieve
	overall family-wise error control
	\begin{equation}\label{ad hoc_typeI}
	\alpha = P\left( \max \limits_{1\leq j \leq n} |X_j| >c \,\, \big| \,\,  \theta_i=0 \,\, \forall i \right) \,.
	\end{equation}
	\begin{lem} (\ref{ad hoc_typeI}) can be expressed as
		\[ \alpha =1- \mathbbm{E}^Z\left\{ \left[\Phi\bigg( \frac{c-\sqrt{\rho}Z}{\sqrt{1-\rho}} \bigg) -\Phi\bigg( \frac{-c-\sqrt{\rho}Z}{\sqrt{1-\rho}} \bigg) \right]^n \right\}\,,\]
		where the expectation is with respect to $Z\sim N(0,1)$.
	\end{lem}
	
	\begin{proof}
		By Lemma \ref{xandz} in the Appendix, under the null model, $X_i$ can be written as $X_i = \sqrt{\rho}Z+\sqrt{1-\rho}Z_i$, where the $Z$ and the $Z_i$ are independent standard normal random variables. Thus
		\begin{equation*}
		\begin{split}
		&P\left( \max \limits_{1\leq j \leq n} |X_j| >c \,\, \big| \,\,  \theta_i=0 \,\, \forall i \right) \\
		&=1 - \mathbbm{E}^Z  \left\{ P\bigg(\mbox{for all $j$}, \,\, |\sqrt{\rho}Z+\sqrt{1-\rho}Z_j | <c \mid Z \bigg)\right\} \\
		&= 1- \mathbbm{E}^Z\left\{\prod\limits^n_1 P\bigg( \frac{-c-\sqrt{\rho}Z}{\sqrt{1-\rho}} <Z_j < \frac{c-\sqrt{\rho}Z}{\sqrt{1-\rho}} \,\, \big| \,\, Z \bigg)\right\} \\
		&=  1- \mathbbm{E}^Z\left\{ \bigg[\Phi\bigg( \frac{c-\sqrt{\rho}Z}{\sqrt{1-\rho}} \bigg) -\Phi\bigg( \frac{-c-\sqrt{\rho}Z}{\sqrt{1-\rho}} \bigg)  \bigg]^n \right\} \,.
		\end{split}
		\end{equation*}
	\end{proof}

	\begin{cor} \text{ }
		\begin{itemize}
			\item 	When $\rho=0$, $\Phi(c) = 1 + \frac{\log(1-\alpha)}{2n}+O(1/n^2)$, essentially calling for the Bonferroni correction.
			\item 	When $\rho \rightarrow 1$, $\Phi(c) \rightarrow 1 -\frac{\alpha}{2}$, so the critical region is
			the same as that for a single test.
		\end{itemize}
	\end{cor}
	\begin{proof}
		If $\rho =0$, $\alpha = 1 - (\Phi(c)-\Phi(-c))^n = 1 - (2\Phi(c)-1)^n$, from which it follow that
		\begin{equation*}
		\Phi(c)= \frac{1+(1-\alpha)^{1/n}}{2} = \frac{1+1 + \frac{\log(1-\alpha)}{n} +O(1/n^2)}{2} \,.
		\end{equation*}
		If $\rho \rightarrow 1$, by Lemma \ref{xandz}, 
		\begin{equation*}
		\begin{split}
		& \lim \limits_{\rho\rightarrow 1 }P\left( \max \limits_{1\leq j \leq n} |X_j| >c \,\, \big| \,\,  \theta_i=0 \,\, \forall i \right) \\
		&=1 - \lim \limits_{\rho\rightarrow 1 }\mathbbm{E}^{Z_1,...,Z_n}  \left\{ P\bigg(|\sqrt{\rho}Z+\sqrt{1-\rho}Z_j | <c \mid Z_1,..,Z_n \bigg)\right\} \\
		&=1 -\mathbbm{E}^{Z_1,...,Z_n}  \left\{  \lim \limits_{\rho\rightarrow 1 } P\bigg(|\sqrt{\rho}Z+\sqrt{1-\rho}Z_j | <c \mid Z_1,..,Z_n \bigg)\right\} \\
		%
		%
		%
		%
		&= 1 - (\Phi(c)-\Phi(-c)) \\&= 2 (1-\Phi(c))\,.
		\end{split}
		\end{equation*}
	\end{proof}
	\noindent
	The extreme effect of dependence on frequentist multiplicity correction is clear here; the correction ranges from full Bonferroni correction to no correction, as the correlation ranges from 0 to 1.
	
	
	\subsection{Likelihood Ratio Test}
	A more principled frequentist procedure would be the likelihood ratio test (LRT):
	\begin{thm} \label{LRT}
		The test statistic arising from the likelihood ratio test is
		\[ T = \max\limits_j \bigg[\sqrt{1-\rho} \ x_j +n\rho\bigg(\frac{x_j -\boldsymbol{\bar{x}}}{\sqrt{1-\rho}}\bigg)\bigg]^2\] 
		and the LRT would be to reject the null hypothesis if $T >c$, where $c$ satisfies
		$ \alpha = P(T>c \mid \theta_i=0 \, \forall i)$.
	\end{thm}
	
	\begin{proof}
		
		Denote
		\[\Sigma_0=
		\begin{pmatrix}
		1 & \rho & \cdots & \rho \\
		\rho & 1 & \cdots & \rho \\
		\vdots  & \vdots  & \ddots & \vdots  \\
		\rho & \rho & \cdots & 1
		\end{pmatrix} \]
		and its inverse
		\[ \Sigma_0^{-1}=
		\begin{pmatrix}
		a & b & \cdots & b \\
		b & a & \cdots & b \\
		\vdots  & \vdots  & \ddots & \vdots  \\
		b & b & \cdots & a
		\end{pmatrix}  \,\, where \,\,
		\begin{cases} a = a_n=\frac{1+(n-2)\rho}{(1+(n-1)\rho)(1-\rho)} \\
		b = b_n =\frac{-\rho}{(1+(n-1)\rho)(1-\rho)} \end{cases} \,.\]
		The likelihood ratio is then, letting $f(\cdot)$ denote the density of $X$ and
		$\mathbf{\tilde{x}_i}= (x_1,...,x_{i-1},x_i-\theta_i,  x_{i+1},...x_n)'$,
		\begin{equation}
		LR = \frac{f(\boldsymbol{x} \mid \theta_i=0 , \forall \, i)}{ \max \limits_{i,\theta_i} f(\boldsymbol{x} \mid \theta_i\neq 0, \theta_{-i}=0)}
		=  \frac{(\det \Sigma_0)^{-\frac{1}{2}} \exp\left\{\frac{-1}{2}\boldsymbol{x}^T \Sigma_0^{-1} \boldsymbol{x} \right\}  }
		{ \max \limits_{i} \, (\det \Sigma_0)^{-\frac{1}{2}} \exp\left\{\frac{-1}{2}\sup_{\theta_i}\mathbf{\tilde{x}_i}^T \Sigma_0^{-1} \mathbf{ \tilde{x}_i}\right\} } \,.  \end{equation}
		Computation yields, defining $u_i = \sum\limits^n_{j\neq i} x_j $,
		\[	\hat{\theta}_i=\argmax \limits_{\theta_i} 	\frac{-1}{2}\mathbf{\tilde{x}_i}^T \Sigma_0^{-1} \mathbf{\tilde{x}_i}=
		x_i + \frac{b}{a}u_i  \,, \]
		from which it is immediate that
		\begin{equation*}
		\begin{split}
		LR &=\min \limits_i \frac{ \exp \left\{\frac{-1}{2} \bigg((a-b)(\sum_1^n x_j^2)+b(\sum\limits^n_{1} x_j)^2 \bigg) \right\}}
		{ \exp \left\{\frac{-1}{2} \bigg((a-b)( \frac{b^2}{a^2}u_i^2 +\sum\limits_{j\neq i}^n x_j^2)+b(-\frac{b}{a}u_i+ \sum\limits^n_{j\neq i} x_j)^2 \bigg) \right\}} 
		\\&=\min \limits_i \,\exp\left\{\frac{-1}{2} \bigg[  (a-b)(x_i^2 -\frac{b^2}{a^2}u_i^2 )
		+b\bigg((u_i+x_i)^2 -u_i^2(\frac{b}{a}-1)^2 \bigg) \bigg]\right\}\\
		& \hspace{7cm}  (\text{since } \sum^n_1 x_j = u_i + x_i  )\\
		&=\min \limits_i \,\exp\left\{\frac{-1}{2} \bigg[ ax_i^2 + 2bu_ix_i + \frac{b^2}{a}u_i^2 \bigg] \right\}
		\\&=\min \limits_i \, \exp\left\{\frac{-1}{2a} ( ax_i + bu_i)^2 \right\} \,.
		\end{split}
		\end{equation*}
		Noting that
		\begin{equation*}
		\begin{split}
		&\frac{1}{a}(ax_j + bu_j)^2 \\&=\frac{1}{(1+(n-1)\rho)(1+(n-2)\rho)(1-\rho)}
		\bigg( (1+(n-2)\rho )x_j - \rho \sum\limits_{k\neq j}x_k \bigg)^2\\
		&= \frac{1}{(1+(n-1)\rho)(1+(n-2)\rho)(1-\rho) }\bigg[(1-\rho)x_j +n\rho(x_j -\boldsymbol{\bar{x}})\bigg]^2\\
		&= \frac{1}{(1+(n-1)\rho)(1+(n-2)\rho) }\bigg[\sqrt{1-\rho}x_j +n\rho\bigg(\frac{x_j -\boldsymbol{\bar{x}}}{\sqrt{1-\rho}}\bigg)\bigg]^2 \,,
		\end{split}
		\end{equation*}
		it is immediate that LR is equivalent to the test statistic $T$.
		
		The rejection region is $LR \leq k$ for some $k$, which is clearly equivalent to $T \geq c$ for appropriate critical value $c$.
	\end{proof}
	When $\rho=0$, $T=\max \limits_i x_i^2$, and the LRT reduces to the ad hoc testing procedure in the previous section. On the other hand,
	as $\rho\rightarrow 1$,  $T  \approx \max \limits_i \,\, n^2(\frac{x_i-\boldsymbol{\bar{x}}}{\sqrt{1-\rho}})^2$, which exhibits a quite different behavior that
	will be discussed later.

	\section{A Bayesian Test}
	
	On the Bayesian side, it is convenient to view this as the model selection problem of
	deciding between the $n+1$ exclusive models
	\begin{align} \label{model}& M_0: \theta_1=...=\theta_n=0\, (null \, model) \nonumber\\
	& M_i: \theta_i \neq 0  , \, \,\theta_{(-i)} = \mathbf{0}
	\end{align}
	where $\theta_{(-i)}$ is the vector of all $\theta_j$ except $\theta_i$.
	
	The simplest prior assumption computationally is that, for the nonzero $\theta_i$ (if there is one),
	\[\theta_i\sim N(0,\tau^2) \,;
	\]
	initially we will assume $\tau^2$ to be known, but later will consider it to be unknown.
	Then under model $M_i$, the marginal likelihood of model $M_i$ is
	\begin{eqnarray} & m_0(\boldsymbol{x}) \sim  N(\mathbf{0}, \Sigma_0) \nonumber\\
	& m_i(\boldsymbol{x})=\int f(\boldsymbol{x} \mid \boldsymbol\theta )\pi(\boldsymbol\theta) d\boldsymbol\theta
	\sim  N\left(   \mathbf{0},  \Sigma_i \right)  \,,
	\end{eqnarray}
	where
	\[  \Sigma_i=\begin{pmatrix}
	1 & \rho & \cdots & \rho \\
	\rho & 1 & \cdots & \rho \\
	\vdots  & \vdots  & 1+\tau^2 & \vdots  \\
	\rho & \rho & \cdots & 1
	\end{pmatrix} \,. \]
	The posterior probability of $M_i$ (that the $i^{th}$ channel has the signal) is then
	\begin{eqnarray*}
		P(M_i \mid \boldsymbol{x}) &=& \frac{m_i(\boldsymbol{x}) P(M_i)}{\sum \limits_{j=0}^n m_j (\boldsymbol{x}) P(M_j)} \,,
	\end{eqnarray*}
	where $P(M_j)$ is the prior probability of model $M_j$.



	\begin{thm} \label{post_full} 
		For any $\rho \in [0,1)$ and positive integer $n>1$, the null posterior probability is:
		\begin{equation*}
		\begin{split}
		P(M_0 \mid \boldsymbol{x})&=
		\left\{  1+ \bigg(\frac{1-r}{n\, r}\bigg) \frac{1}{\sqrt{1+\tau^2 a}} \sum \limits^n_{i=1}\exp \left\{\frac{\tau^2}{2(1+\tau^2a)} \bigg( \frac{x_i}{1-\rho} +bn\boldsymbol{\bar{x}} \bigg)^2 \right\}
		\right\}^{-1}
		\end{split} \,,
		\end{equation*}
		
		and the posterior probability of an alternative  model $M_i$ is
		\label{eq:post_prob_simplfy}
		\begin{equation}
		P(M_i \mid  \boldsymbol{x})=
		\left\{ \begin{array}{l}
		\sqrt{1+a\tau^2}  (\frac{n\, r}{1-r})\exp \left\{ \frac{-\tau^2}{2(1+\tau^2 a)}\bigg(\frac{x_i}{1-\rho} +n\boldsymbol{\bar{x}}b \bigg)^2 \right\} \\
		+\sum \limits_{k=1}^n \exp \left\{ \frac{-\tau^2}{2(1+\tau^2a)}\bigg( \frac{x_i^2-x_k^2}{(1-\rho)^2}+\frac{2b}{1-\rho} (n\boldsymbol{\bar{x}})(x_i-x_k) \bigg) \right\}
		\end{array}
		\right\}^{-1} \,.
		\end{equation}
		
	\end{thm}

	\begin{proof}
		The posterior probability of $M_i$ is
		\small
		\begin{equation}\label{eq:posterior}
		\begin{split}
		& P(M_i \mid \boldsymbol{x})=\frac{m_i(\boldsymbol{x})P(M_i)}{\sum \limits_{j=0}^n m_j(\boldsymbol{x}) P(M_j)}\\
		&=
		\frac{\frac{1-r}{n}|\Sigma_i|^\frac{-1}{2} exp \left\{\frac{-1}{2}\boldsymbol{x}' \Sigma_i^{-1} \boldsymbol{x} \right\}}
		{ r|\Sigma_0|^\frac{-1}{2} exp \left\{\frac{-1}{2}\boldsymbol{x}' \Sigma_0^{-1}\boldsymbol{x} \right\} +
			\sum \limits ^n_{j=1}\frac{1-r}{n}|\Sigma_j |^\frac{-1}{2} exp \left\{\frac{-1}{2} \boldsymbol{x}' \Sigma_j^{-1} \boldsymbol{x} \right\} }  \\
		&=\left\{ \begin{array}{l}
		\bigg( \frac{n \, r}{1-r} \bigg)  \bigg| \frac{\Sigma_0}{\Sigma_1} \bigg|^\frac{-1}{2}
		exp \left\{\frac{-1}{2}\boldsymbol{x}' (\Sigma_0^{-1}-\Sigma_i^{-1})\boldsymbol{x} \right\} \\
		+  1+ 
		\sum \limits ^n_{j \neq i}  exp \left\{\frac{-1}{2}\boldsymbol{x}' (\Sigma_j^{-1}-\Sigma_i^{-1})\boldsymbol{x} \right\}
		\end{array} \right\}^{-1}\,.
		\end{split}
		\end{equation}
		
		The expression can be simplified by further computing $\Sigma_i^{-1}$,$(\Sigma_i^{-1}-\Sigma_k^{-1})$  and $det(\Sigma_i)$. First notice that by the Cholesky decomposition
		\[ \Sigma_0^{-1}=
		\begin{pmatrix}
		a & b & \cdots & b \\
		b & a & \cdots & b \\
		\vdots  & \vdots  & \ddots & \vdots  \\
		b & b & \cdots & a
		\end{pmatrix} =\boldsymbol{L}\boldsymbol{L}^T \,,\]
		for some lower triangular matrix $\boldsymbol{L}$. Then by the Woodbury identity, the difference of two inverse matrices can be obtained:
		\begin{eqnarray*}
			\Sigma_i^{-1} =(\Sigma_0+\boldsymbol{\tau}_i \boldsymbol{\tau}_i ^T)^{-1}
			&=&\Sigma_0^{-1}-\Sigma_0^{-1}\boldsymbol{\tau}_i(1+\boldsymbol{\tau}_i^T\Sigma_0^{-1}\boldsymbol{\tau}_i)^{-1}\boldsymbol{\tau}_i^T\Sigma_0^{-1} \\
			&=&\Sigma_0^{-1}\frac{-\tau^2}{1+\tau^2a} \Sigma_0^{-1}
			\begin{pmatrix}
				& & & \makebox(0,0){\text{\huge0}} & & & \\
				b & \cdots &b &a &b &\cdots &b \\
				& & & \makebox(0,0){\text{\huge0}} & & & \\
			\end{pmatrix}\,,
		\end{eqnarray*}
			where $\boldsymbol{\tau}_i =(0, \cdots, \tau , \cdots, 0)^T   \qquad \text{(the $i^{th}$ element is $\tau^2$)}$.
				Therefore, 				
		\begin{eqnarray*}
			\boldsymbol{x}' (\Sigma_0^{-1}-\Sigma_i^{-1})\boldsymbol{x} &=& \frac{\tau^2}{1+\tau^2a} \bigg( x_i(a-b)+bn\boldsymbol{\bar{x}} \bigg)^2=
			\frac{\tau^2}{1+\tau^2a} \bigg( \frac{x_i}{1-\rho}+bn\boldsymbol{\bar{x}} \bigg)^2 \\
			\boldsymbol{x}'(\Sigma_k^{-1}-\Sigma_i^{-1})\boldsymbol{x} &=&  \frac{\tau^2}{1+\tau^2a}(x_i^2-x_k^2)(a-b)^2+2b(a-b)(n\boldsymbol{\bar{x}})(x_i-x_k)\,.
		\end{eqnarray*}
				Also the ratio of two determinants is 
				\begin{equation*} \begin{split}
		\det (\Sigma_i) /\det(\Sigma_0) &=\det(\Sigma_1)/det(\Sigma_0)
		= \det(I+\Sigma_0^{-1}\boldsymbol{\tau_1} \boldsymbol{\tau_1}i^T)\\
		&= \det(I+\boldsymbol{L}\boldsymbol{L}^T \boldsymbol{\tau_1} \boldsymbol{\tau_1}^T)
		=\det(I+\boldsymbol{\tau_1}^T\boldsymbol{L}^T \boldsymbol{L}\boldsymbol{\tau_1})\\
		&=(1+  \tau^2  L_{11}^2)
		=(1+\tau^2 a)\,.\\
		\end{split} \end{equation*}
				By plugging back these quantities into (\ref{eq:posterior}), the proof is complete.		
	\end{proof}
	
	\begin{rem}
		(\ref{eq:post_prob_full}) gives the full expression for $P(M_i \mid  \boldsymbol{x})$, without using $a,b$, and
		is utilized in the subsequent proofs.
	\end{rem}
	\begin{cor}\label{post_prob_dim2}
		In particular, when $n=2$, the null posterior probability is:
		\begin{equation}
		\begin{split}
		&P(M_0 \mid \boldsymbol{x})= \\&
		\left\{1+\frac{1-r}{2r}\sqrt{\frac{1-\rho^2}{1-\rho^2+\tau^2} }
		\sum_{i\in\{1,2\}}\exp \left\{\frac{\tau^2}{2(1-\rho^2+\tau^2)}
		\bigg(\frac{x_i-\rho x_{(-i)}}{\sqrt{1-\rho^2}}\bigg)^2 \right\}\right\}^{-1}\,,\\
		&\text{ and the posterior probability of the alternative $M_i, i\in\{1,2\}$ is:}\\
		&P(M_i \mid   \boldsymbol{x}) =
		\left\{
		\begin{array}{l}
		\sqrt{\frac{1-\rho^2+\tau^2}{1-\rho^2}}\bigg(\frac{2r}{1-r}\bigg)
		\exp \left\{\frac{-\tau^2}{2(1-\rho^2+\tau^2)}
		\frac{(x_i-\rho x_{(-i)})^2 }{1-\rho^2} \right\} \\
		+1+\exp \left\{\frac{-\tau^2}{2(1-\rho^2+\tau^2)}
		(x_i^2-x^2_{(-i)}) \right\} \end{array}
		\right\}^{-1}\,. \nonumber	
		\end{split}
		\end{equation}
	\end{cor}
	
	\section{The situation as the correlation goes to 1 }
\label{sec.power}
	The following theorem shows the surprising result that, when  the dimension is greater than $2$, the Bayesian method can correctly select the true model when the correlation goes to one. In two dimensions, however, there is nonzero probability of choosing the wrong alternative model if a non-null model is true.

	\begin{thm}  \label{post_prob_rho1}
		\text{\\}
		
		If $n=2$, $i\in \{1,2\}$ and $\rho\rightarrow 1$, then:
			\begin{equation*}
			\begin{split}
			P(M_0 \mid \boldsymbol{X}) &\rightarrow 1 \quad \text{ under the null model}\,,\\
			P(M_i \mid \boldsymbol{X}) &\rightarrow
			%
			\bigg(1+\exp\left\{\frac{-1}{2} \big(X_i^2-X^2_{(-i)} \big)\right\}\bigg)^{-1}    \text{ under $M_1$ or $M_2$}\,.   	
			%
			\end{split}
			\end{equation*}
			
		If $n>2$, $i,j \in\{0,1,...,n\}$ and $\rho \rightarrow 1$,  under model $M_j$:
			\begin{equation*}
			P(M_{i} |\boldsymbol{X}) \rightarrow \delta^{j}_{i} = \begin{cases}
			1 \mbox{ if } i=j,\\
			0 \mbox{ else}\,.
			\end{cases}
			\end{equation*}
	\end{thm}
	
	\begin{proof}
		By Lemma \ref{xandz}, $x_i$ can be written as $x_i = \theta_i + \sqrt{\rho}z+\sqrt{1-\rho}z_i$. \\
		\textit{Case I: n=2}: 	
		\begin{equation}\label{prob_dim2_dominated_term}
		\begin{split}
		\frac{x_i - \rho x_{(-i)}}{\sqrt{1-\rho^2}} &=
		\frac{(\theta_i - \theta_{(-i)}) +\sqrt{1-\rho}(z_i-\rho z_{(-i)})  +z\sqrt{\rho}(1-\rho) }{\sqrt{1-\rho^2}}\\
		&= \frac{\theta_i - \theta_{(-i)}}{\sqrt{1-\rho^2}} +
		\frac{z_i-\rho z_{(-i)}}{\sqrt{1+\rho}}
		+
		\sqrt{1-\rho}\frac{\sqrt{\rho}z}{\sqrt{1+\rho}}\,.
		\end{split}
		\end{equation}
		
		If $M_0$ is true, since both $\theta_i,\theta_{(-i)}$ are zero, the dominant term of (\ref{prob_dim2_dominated_term}) is $O(1)$, so the null posterior probability (Corollary \ref{post_prob_dim2}) becomes:
		\begin{equation*}
		P(M_0 \mid \boldsymbol{x}) = (1+ O(\sqrt{1-\rho^2})) = 1+ o(1)\,.
		\end{equation*} 	
		
		If $M_i$ is true, the dominant term of (\ref{prob_dim2_dominated_term}) is $ O(1 / \sqrt{1-\rho^2})$; hence as $\rho\rightarrow 1$: 					\begin{equation*}
		\begin{split}
		P(M_i \mid   \boldsymbol{x})
		&= \left\{ \begin{array}{l}
		\sqrt{\frac{\tau^2}{2(1-\rho^2)}}\bigg(\frac{2r}{1-r}\bigg)
		\exp \left\{\frac{-1}{2}
		\frac{\theta_i^2 }{2(1-\rho^2)}+O(\frac{1}{\sqrt{1-\rho}})\right\}
		\\+ 1+\exp \left\{\frac{-1}{2}
		\theta_i(\theta_i+2z)  \right\}\end{array} \right\}^{-1}  \\
		&=
		\bigg( \frac{1}{1+\exp\{\frac{-1}{2}\theta_i (\theta_i+2z)\}} \bigg)(1+o(1)) \\
		&= \frac{1}{1+\exp\left\{\frac{-1}{2} \big(x_i^2-x^2_{(-i)} \big)\right\}}(1+o(1))\,,
		\end{split}
		\end{equation*}
		using the fact that $e^{-1/\sqrt{1-\rho}}$ goes to zero faster than $1/\sqrt{1-\rho}$ goes to infinity.
		
		If $M_{(-i)}$ is true: similar to the previous case, (\ref{prob_dim2_dominated_term}) is $O(1/\sqrt{1-\rho})$, so only the last term in (\ref{post_prob_dim2}) remains.\\
		%
		\textit{Case II: $n>2$}:  denote
		$\boldsymbol{z} =(z_1,...,z_n)$, $\boldsymbol{\bar{z}} = 1/n\sum^n_1 z_i$;
		
		If $M_0$ is true :
		take $\boldsymbol{\theta}=(\theta_1,...,\theta_n)=\boldsymbol{0}$ in (\ref{eq:post_prob_full}), let $c=c(\rho)=$ \\$\frac{1+(n-1)\rho}{(1-\rho+\tau^2)(1+(n-1)\rho)-\tau^2\rho} $,
		and note that $\lim \limits_{\rho \rightarrow 1} c(\rho)=\frac{n}{\tau^2(n-1)}\,.$ Then
		\small
		\begin{equation*}
		\begin{split}
		& P(M_i \mid  \boldsymbol{x}) \\
		&=
		\left\{ \begin{array}{l}
		\sqrt{\frac{1}{c(1-\rho)}}
		\bigg(\frac{n \, r}{1-r}\bigg)*\\
		\exp \left\{ \frac{-\tau^2}{2}c \bigg( (z_i-\boldsymbol{\bar{z}}) +
		\underbrace{\frac{\sqrt{1-\rho}}{1+(n-1)\rho}(\sqrt{\rho}z+\sqrt{1-\rho}\boldsymbol{\bar{z}})}_\textrm{O($\sqrt{1-\rho})$} \bigg)^2 \right\} +1+\\
		\sum \limits^n_{k \neq i} \exp \left\{ \frac{-\tau^2}{2} c
		\left\{ \begin{array}{l}
		\frac{(z_i+z_k-2\boldsymbol{\bar{z}})(z_i-z_k)}{1-\rho}  +\\ 2\underbrace{\frac{(1-\rho)\boldsymbol{\bar{z}}(z_i-z_k)+\sqrt{\rho(1-\rho)}z(z_i-z_k) }{1+(n-1)\rho}}_\textrm{$O(\sqrt{1-\rho})$} \end{array} \right\} \right\}
		\end{array} \right\}^{-1}\\
		&=
		\left\{ \begin{array}{l}
		\sqrt{\frac{\tau^2(n-1)/n}{1-\rho}}  \bigg(\frac{n \, r}{1-r}\bigg)   \exp \left\{ \frac{-n}{2(n-1)}(z_i-\boldsymbol{\bar{z}})^2 \right\}+1+\\
		\sum \limits^n_{k \neq i} \exp \left\{ \frac{-n}{2(n-1)} \frac{(z_i+z_k-\boldsymbol{\bar{z}})(z_i-z_k)}{1-\rho} \right\}
		\end{array} \right\}^{-1} (1+o(1)) \\
		&\leq 
		\left\{
		\sqrt{\frac{\tau^2(n-1)/n}{1-\rho}}  \bigg(\frac{n \, r}{1-r}\bigg)   \exp \left\{ \frac{-n}{2(n-1)}(z_i-\boldsymbol{\bar{z}})^2 \right\}
		\right\} ^{-1} (1+o(1)) \\
		&= \sqrt{1-\rho} \sqrt{\frac{1-r}{r \, \tau^2}}(1+o(1)) \rightarrow 0\,.
		\end{split}
		\end{equation*}
		\normalsize
		
		If $M_j$ is true ($j>0$): Take $\theta_k=0 \,\forall \,k\neq j$ in (\ref{eq:post_prob_full}):
		\begin{equation*}
		\begin{split}
		&P(M_j \mid \boldsymbol{x}) \\&=
		\left\{ \begin{array}{l}
		\sqrt{\frac{1}{c(1-\rho)}}\bigg(\frac{nr}{1-r}\bigg)
		\exp \left\{\frac{-\tau^2 c}{2(1-\rho)}\bigg( (1-\frac{2}{n})\theta_j+\sqrt{1-\rho}(z_i-\boldsymbol{\bar{z}})+O(1-\rho) \bigg)^2   \right\}\\
		+1+\sum \limits_{k \neq j}^n \exp \left\{\frac{-\tau^2c}{2} \bigg( \frac{ (1-2/n) \theta_j +\sqrt{1-\rho}(z_j+z_k-2\boldsymbol{\bar{z}})
			(\theta_j +\sqrt{1-\rho}(z_i-z_k))}{1-\rho}
		+O(1) \bigg) \right\}
		\end{array} \right\}^{-1}  \\ &=
		\left\{ \begin{array}{l}
		\sqrt{\frac{\frac{n-1}{n}\tau^2}{1-\rho}}  \bigg(\frac{n \, r}{1-r}\bigg)
		\exp \left\{ \frac{-n}{2(n-1)}\bigg( \frac{ (1-2/n)^2\theta_j^2}{1-\rho} +O(\frac{1}{\sqrt{1-\rho}}) \bigg) \right\} +1 \\
		+(n-1) \exp\left\{  \frac{- n}{2(n-1)} \bigg( \frac{(1-2/n)\theta_j^2}{1-\rho}+O(\frac{1}{\sqrt{1-\rho}}) \bigg) \right\}
		\end{array} \right\}^{-1}(1+o(1)) \\
		&\rightarrow 1 \text{ since } \lim_{\rho\rightarrow 1}\sqrt{1/(1-\rho)}\exp\{-1/(1-\rho)\} = 0.
		\end{split}
		\end{equation*}
		%
		%
	\end{proof}

	\begin{thm} \label{LRT_rho_goes_to_1}
		The likelihood ratio test (Theorem \ref{LRT}) is fully powered (i.e., rejects the null with probability 1 under an alternative hypothesis) when $\rho \rightarrow 1$ and $n>2$, but (as with the Bayesian test) is not fully powered when $n=2$.
	\end{thm}
	
	\begin{proof}
		By Lemma \ref{xandz2}:
		\begin{equation*}
		\begin{cases}
		\frac{x_i -\boldsymbol{\bar{x}}}{\sqrt{1-\rho}} =z_i-\boldsymbol{\bar{z}}   &\mbox{ under the null model } M_0\,,\\
		\frac{x_i -\boldsymbol{\bar{x}}}{\sqrt{1-\rho}} =\frac{\theta_j(\delta^i_j-1/n)}{\sqrt{1-\rho}}+z_i-\boldsymbol{\bar{z}}
		&\mbox{ under an alternative model } M_j\,.
		\end{cases}
		\end{equation*}
		When $n=2$, under $M_i$, when $\rho\rightarrow 1$:	
		\begin{equation*}
			\lim_{\rho\rightarrow 1}T = \lim_{\rho\rightarrow 1} \max_{j \in \{1,2\}} \bigg[\sqrt{1-\rho}x_{j} +2\rho \bigg(\frac{x_{j} -\boldsymbol{\bar{x}}}{\sqrt{1-\rho}}\bigg)^2\bigg]
		= 2 \lim_{\rho\rightarrow 1}
		\max_{j \in  \{1,2\}}\rho\bigg( \frac{x_j-x_{(-j)}}{2\sqrt{1-\rho}}\bigg)^2\,.
		%
		%
			\end{equation*}	
		For both $j=i$ or $j=(-i)$, the corresponding likelihood ratios go to infinity at the same asymptotic rate since
		$ \big[\theta_i/(2\sqrt{1-\rho})+(z_i-z_{(-i)})/2  \big]^2 =\big[-\theta_i/(2\sqrt{1-\rho})-(z_i-z_{(-i)})/2  \big]^2$.
		When  $n>2$, under $M_j$, when $\rho\rightarrow 1$:	
		\begin{equation*}
		\begin{split}
		\max_i \bigg[\sqrt{1-\rho}x_i +n\rho \bigg(\frac{x_i -\boldsymbol{\bar{x}}}{\sqrt{1-\rho}}\bigg)^2\bigg]
		&= \max_i n\bigg[\frac{\theta_i-\theta_j/n}{\sqrt{1-\rho}}+z_i-\boldsymbol{\bar{z}}  \bigg]^2+o(1) \\
		&= n\bigg[\frac{\theta_j(1-1/n)}{\sqrt{1-\rho}}+z_j-\boldsymbol{\bar{z}}  \bigg]^2+o(1)\,.
		\end{split}
		\end{equation*}
		In this case, the true alternative model has largest likelihood ratio (=$\infty$), hence, LRT is fully powered.
	\end{proof}
	
	From Theorem \ref{post_prob_rho1} and \ref{LRT_rho_goes_to_1}, when the correlation goes to 1 and the dimension is larger than 2, both the Bayesian procedure and the LRT are fully powered. This surprising behavior as the correlation goes to one can be explained by the following observations using (\ref{xandz}).
	
	When $n=2$, $\rho\rightarrow1$:
\begin{equation*}
\begin{split}
	 x_i-x_j &= [\theta_i+\sqrt{\rho}z+\sqrt{1-\rho}z_i]-[\theta_j+\sqrt{\rho}z+\sqrt{1-\rho}z_j]
	 =\begin{cases} 0 \mbox{    under } M_0   \\
	 \theta_i \, \, or\,  -\theta_j  \mbox{ else}\,.  \end{cases}
\end{split}
\end{equation*}
	Hence, one can correctly distinguish the null model if it is true, but can not declare which non-null model is true when $x_i-x_j$ is not 0.
	
	When $n>2$, $\rho \rightarrow 1$:
	if all pairs $x_i-x_j$ are zero, then the null model is true. If there are pairs $x_i-x_j$, $x_j-x_i$ that are nonzero,
	we can further check whether $x_i - x_k$ ($k\neq j$) equals zero or not to see whether $\theta_i$ or $\theta_j$ is nonzero.
	
	Note that the ad hoc frequentist test does not have this behavior. As $\rho \rightarrow 1$, the test
	\begin{itemize}
		\item still has probability $\alpha$ of incorrectly rejecting a true $M_0$;
		\item still has positive probability of not detecting a signal when $M_i$ is true.
	\end{itemize}
This highlights the danger (in terms of lack of power) of using `intuitive' procedures for multiplicity control.


	\section{Asymptotic frequentist properties of Bayesian procedures}
	
	In this section, we will be studying the false positive probability (FPP) theoretically and numerically. We first need to obtain asymptotic posteriors.
	

	\subsection{Posterior probabilities}

	\begin{lem} \label{lemma:post_n_infinity}
		As n $\rightarrow \infty$ under the null model, 
		\small
		\begin{equation} \label{post_n_infinity}
		\begin{split} 
		P(M_i \mid  \boldsymbol{x})
		&= \bigg(1 + \frac{n}{1-r} \sqrt{\frac{1-\rho+\tau^2}{1-\rho}}
		\exp\left\{\frac{-\tau^2}{2(1-\rho+\tau^2)}\bigg(\frac{x_i-\boldsymbol{\bar{x}}} {\sqrt{1-\rho}}\bigg)^2 \right\} \bigg)^{-1}(1+o(1))
		\end{split}
		\end{equation}
		almost surely.
	\end{lem}
	\normalsize
	\begin{proof}
		Take $\boldsymbol{\theta}=0$ in (\ref{eq:post_prob_full}):
		\small
		\begin{equation} \label{post_n_infty_proof}
		\small
		\begin{split}
				&P(M_i \mid \boldsymbol{x})=\\  &\left\{
		\begin{array}{l}
		\sqrt{\frac{1}{c(1-\rho)}} (\frac{n \, r}{1-r})
		\exp \left\{\frac{-\tau^2}{2}c
		\bigg( \underbrace{ z_i-\boldsymbol{\bar{z}}+\frac{\sqrt{1-\rho}}{1+(n-1)\rho} (\sqrt{\rho}z+\sqrt{1-\rho}\boldsymbol{\bar{z}}) \bigg)^2}_\textrm{$I$} \right\} + 1 + \\
		\sum\limits^n_{k\neq i } \exp \{\frac{-\tau^2}{2}c
		\bigg( \underbrace{(z_i+z_k-2\boldsymbol{\bar{z}})(z_i-z_k)+2\frac{(\sqrt{\rho}z+\sqrt{1-\rho}\boldsymbol{\bar{z}})(\sqrt{1-\rho}(z_i-z_k))}{1+(n-1)\rho} }_\textrm{$II$} \bigg) \}
		\end{array} \right\}^{-1} \\
		\end{split}
		\end{equation}
		\normalsize
		Without loss of generality, assuming
		$|z_i| \leq n^{1/2-\epsilon}$ for all $i$, which holds almost surely by Lemma \ref{x_upper_bdd}, asymptotic analysis of (\ref{post_n_infty_proof}) yields:
		\begin{equation*} \label{eq:proof_lemma3}
		\begin{split}
		I  
		&=  \bigg( \big(1-1/n\big)^2
		z_i^2+ \underbrace{\boldsymbol{\bar{z}}^2_{(-i)}}_\textrm {O($n^{-1}$)}  -2(1-1/n)\underbrace{z_i\boldsymbol{\bar{z}}_{(-i)}}_\textrm{$ O(n^{-\epsilon})$}
		%
		%
		\bigg) \,\, (\text{ where $\boldsymbol{\bar{z}}_{(-i)}= 1/n\sum_{k\neq i} z_k$}) \\
		&\qquad +2\underbrace{\bigg(\frac{\sqrt{1-\rho}}{1+(n-1)\rho}\bigg)}_\textrm{$O(n^{-1})$}
		\underbrace{\bigg((1-1/n)z_i-\boldsymbol{\bar{z}}_{(-i)}\bigg)}_\textrm{$O(n^{1/2-\epsilon})$}
		\underbrace{\bigg(\sqrt{\rho}z+\sqrt{1-\rho}\boldsymbol{\bar{z}} \bigg)}_\textrm{$O(1)$} \\
		&\qquad +\underbrace{\bigg(\frac{\sqrt{1-\rho}}{1+(n-1)\rho} \bigg)^2}_\textrm{$O(n^{-2})$}
		\underbrace{(\sqrt{\rho}z+\sqrt{1-\rho}\boldsymbol{\bar{z}})^2}_\textrm{$O(1)$} \\
		&=\left(1-1/n\right)^2z_i^2+O(n^{-\epsilon})\,.
		\end{split}
		\end{equation*}
				Therefore,
		\begin{equation} \label{proof_post_n_infinity_term1} \begin{split}
		&\sqrt{\frac{1}{c(1-\rho)}}  \bigg(\frac{n \, r}{1-r} \bigg) \exp\left\{ 		\frac{-\tau^2}{2}c \,  I \right\} \\
		&= \sqrt{\frac{1}{c(1-\rho)}} \bigg(\frac{n \, r}{1-r} \bigg)\exp \left\{ 	\frac{-\tau^2}{2}c\left( 1-1/n\right)^2 z_i^2  \right\}
		\bigg(1+O(n^{-\epsilon})\bigg) \\
		&=  \sqrt{\frac{1-\rho+\tau^2}{1-\rho}} \bigg(\frac{n \, r}{1-r} \bigg) \exp 	\left\{ \frac{-\tau^2}{2(1-\rho+\tau^2)} z_i^2  \right\}
		(1+o(1))\,.
		\end{split} \end{equation}

		\begin{equation*} \begin{split}
		\uppercase\expandafter{\romannumeral2} 
		%
		%
		&=\big( z_i^2-
		\underbrace{z_k^2}_\textrm{$O(n^{1-2\epsilon})$} \big)\big(1-2/n\big)
		-\underbrace{\big( 2/n\sum_{l\notin\{i,k\}}z_l\big)}_\textrm{$O(n^{-1/2})$}
		\underbrace{(z_i-z_k)}_\textrm{$O(n^{1/2-\epsilon})$}\\
		&\hspace{.5 cm} +\underbrace{\bigg(\frac{2}{1+(n-1)\rho}\bigg)}_\textrm{$O(n^{-1})$}
		\left\{
		\begin{array}{l}
		\sqrt{\rho(1-\rho)}\underbrace{z_iz}_\textrm{$O(n^{1/2-\epsilon})$}
		+(1-\rho)\underbrace{z_i/n(-z_k)}_\textrm{$O(n^{-1/2-\epsilon})$}\\
		+(1-\rho)\underbrace{z^2_i/n}_\textrm{$O(n^{-2\epsilon})$}+
		\underbrace{(\sqrt{\rho}z+\sqrt{1-\rho} \boldsymbol{\bar{z}}_{(-i)})(-z_k)}_\textrm{$O(n^{1/2-\epsilon})$}
		\end{array} \right\} \\
		&=(z_i^2-z_k^2)\left(1-2/n\right)+O(n^{-\epsilon})\,.
		%
		%
		\end{split} \end{equation*}
				The summation term in (\ref{post_n_infty_proof}) becomes: \small
		\begin{equation} \label{proof_post_n_infinity_term2}
		\begin{split}
		&\sum \limits_{k\neq i}^n  \exp \left\{ \frac{-\tau^2}{2}c \,  \uppercase\expandafter{\romannumeral2} \right\} \\
		%
		&= \exp \left\{ -c\frac{\tau^2}{2} (1- 2/n)z_i^2 \right\}
		\left(1+O(n^{-\epsilon})\right) \sum \limits^n_{k=1} \exp \left\{ c\frac{\tau^2}{2}(1-2/n)z_k^2  \right\}  \\
		&= \exp \left\{ \frac{-\tau^2}{2(1-\rho+\tau^2)} z_i^2 \right\}
		n \bigg(\mathbbm{E}^{Z} \bigg[ \exp \left\{ \frac{\tau^2}{2(1-\rho+\tau^2)}Z^2  \right\}  \bigg]+o(1) \bigg)(1+o(1))\\
		& \qquad \qquad \text{ by the Law of Large Numbers and since } Z\sim N(0,1)\\
		&= \exp \left\{ \frac{-\tau^2}{2(1-\rho+\tau^2)} z_i^2 \right\}  n \sqrt {\frac{1-\rho+\tau^2}{1-\rho}}(1+o(1))\,.
		\end{split} \end{equation}\normalsize
		The proof is completed by adding (\ref{proof_post_n_infinity_term1}) and (\ref{proof_post_n_infinity_term2}).
	\end{proof}
	\begin{rem}
		Note that, by Lemma \ref{xandz2},
		\begin{equation} \label{eq:zandx}
		z_i = z_i(\boldsymbol{x})=\frac{x_i-\boldsymbol{\bar{x}}}{\sqrt{1-\rho}}+ O(1/\sqrt{n})\,,\end{equation}
		so that Lemma \ref{lemma:post_n_infinity} can be written, with respect to $z_i$:
		\[P(M_i \mid  \boldsymbol{x})
		= \bigg(1 + \frac{n}{1-r} \sqrt{\frac{1-\rho+\tau^2}{1-\rho}}
		\exp\left\{\frac{-\tau^2 z_i^2}{2(1-\rho+\tau^2)} \right\}(1+o(1)) \bigg)^{-1}\,.
		\]
	\end{rem}

	\begin{rem}
		Figure~\ref{fig:P(M1|X)_n_grows} 
		shows the ratio of the estimated $P(M_1\mid \boldsymbol{x})$ (from Lemma \ref{post_n_infinity}) and the true probability (from Theorem \ref{eq:post_prob_simplfy}), as $n$ grows. Each plot contains $200$ different ratio curves based on independent simulations with fixed $\rho,P(M_0)$ and $\tau$. As can be seen, the ratio goes to $1$ when $n$ grows and the convergence rate indeed depends on the correlation.
	\end{rem}
	
	\begin{figure}[H]
		\hbox{\hspace{-5ex}
			\includegraphics[scale=.5] {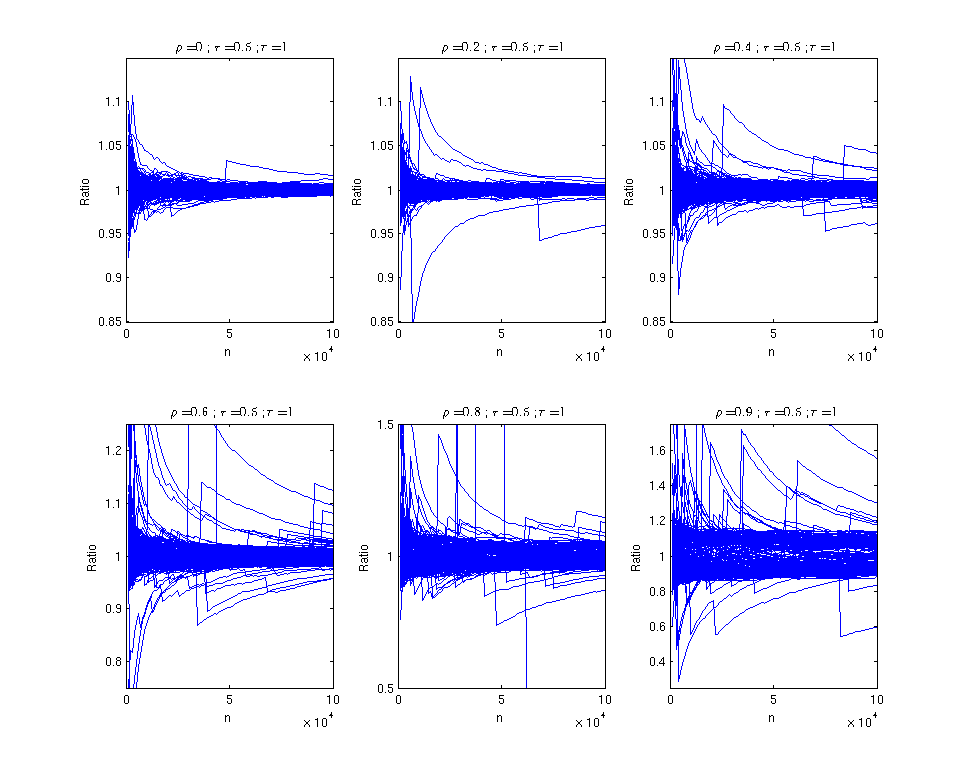} 
		}
		\caption{Ratio of estimated and true posterior probability of $M_1$ as $n$ grows under the null model and fixed $\tau, r$, different $\rho$. Each subplot is for different correlations and contains $200$ simulations.}
		\label{fig:P(M1|X)_n_grows}
	\end{figure}

	\begin{rem} Figure~\ref{fig:P(M1|X)_tau_grows} gives the estimated and true posterior probability of $M_1$ under the assumption that the null model is true, for fixed $r=\rho=0.5$ and $n$, but varied $\tau$. Notice that, for fixed $n$, the estimated probability is closer to the true probability when $\tau$ is small but is worse for larger $\tau$, indicating that larger $n$ is required for obtaining the same precision as $\tau$ grows.
	\end{rem}
	
	\begin{figure}[H]
		\hbox{\hspace{-12ex}
			\includegraphics[trim = 0  120  0 4 ,clip,  width=15cm, height=10cm]{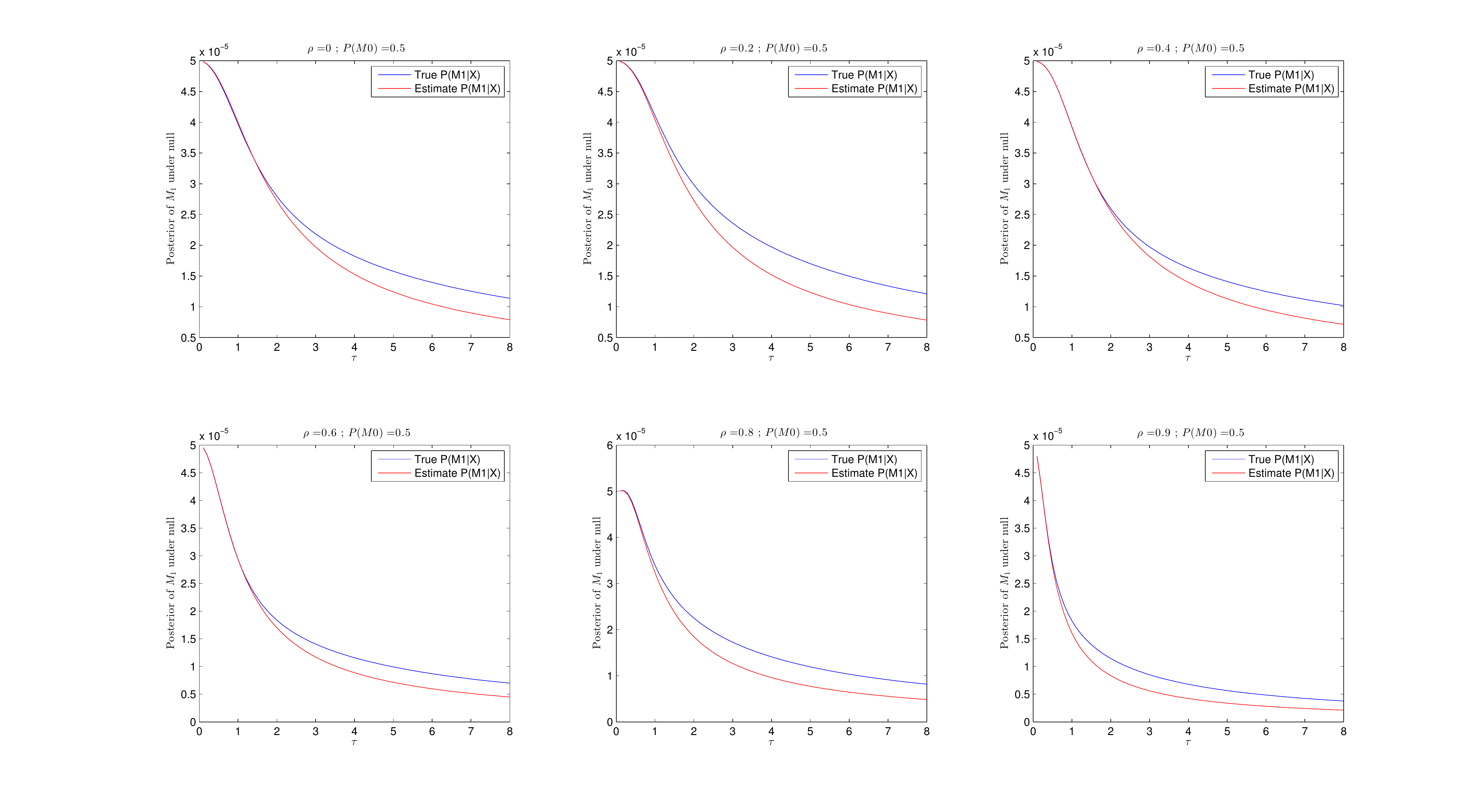}
		}
		\caption{Estimated (red line) and true posterior probability (blue line) of $M_1$ for different $\tau$ under the null model, for fixed $n=2000, \rho=r=0.5$. }
		\label{fig:P(M1|X)_tau_grows}
	\end{figure}

	The following theorem shows the surprising result that, as $n$ grows when the null model is true, the posterior
	probability of the null model converges to its prior probability. Thus one cannot learn that the null model is true.
	\begin{thm} \label{P(M0|X)=r}
		As $n \rightarrow \infty $ and $\rho \in [0,1)$, under the null model,
		\begin{equation*} P(M_0 \mid  \boldsymbol{X}) \rightarrow P(M_0)\,. \end{equation*}
	\end{thm}
	
	\begin{proof}
		First note that 
		\begin{equation}
		\begin{cases}
		a_n =\frac{1}{1-\rho} +\frac{-\rho}{(1-\rho)(1+(n-1)\rho)} =\frac{1}{1-\rho} +O(1/n)\,, \\
		n b_n =\frac{-1}{1-\rho} +\frac{1-\rho}{(1-\rho)(1+(n-1)\rho)} =\frac{-1}{1-\rho} +O(1/n)\,.
		\end{cases}
		\end{equation}
		The summation term in the null posterior (Theorem \ref{post_full}) becomes
		\begin{equation*}
		\begin{split}
		&\bigg(\frac{1-r}{nr}\bigg) \frac{1}{\sqrt{1+\tau^2/(1-\rho)}}
		\sum^n_1 \exp\left\{\frac{\tau^2}{2(1+\tau^2/(1-\rho) ) }
		\bigg[ \frac{x_i-\boldsymbol{x}}{1-\rho} \bigg]^2  \right\}(1+o(1)) \\
		%
		%
		&= \bigg(\frac{1-r}{r}\bigg) 1/n \sqrt{\frac{1-\rho}{1-\rho+\tau^2}}
		\sum^n_1 \exp\left\{\frac{\tau^2}{2(1-\rho+\tau^2 ) } z_i^2 \right\}(1+o(1)) \,\, (\text{by Lemma } \ref{xandz2})\\
		&\rightarrow \frac{1-r}{r} \quad \text{ (by the Strong Law of Large Numbers).}	
		\end{split}
		\end{equation*}
		
		Therefore,
		$P(M_0 \mid  \boldsymbol{X}) \rightarrow \big(1+ (1-r)/r\big)^{-1} = r = P(M_0)\,.$

	\end{proof}

	\begin{rem}
		Figure~\ref{fig:P(M0|X)=r} shows simulations of the null posterior probability for different numbers of hypotheses and different correlations. 
		Interestingly, by Theorem \ref{post_prob_rho1}, the Bayes procedure identifies the correct model (here the null model) when $n$ is fixed and the correlation goes to 1, resulting in higher initial posterior probability of the null model for highly correlated cases. On the other hand, by Theorem \ref{P(M0|X)=r}, this posterior probability converges to its prior probability regardless of the correlation. This convergence can be seen in Figure~\ref{fig:P(M0|X)=r}.
	\end{rem}
	
	\begin{figure}[H]
		
		\hbox{\hspace{-5ex}
			\includegraphics[scale=.31]{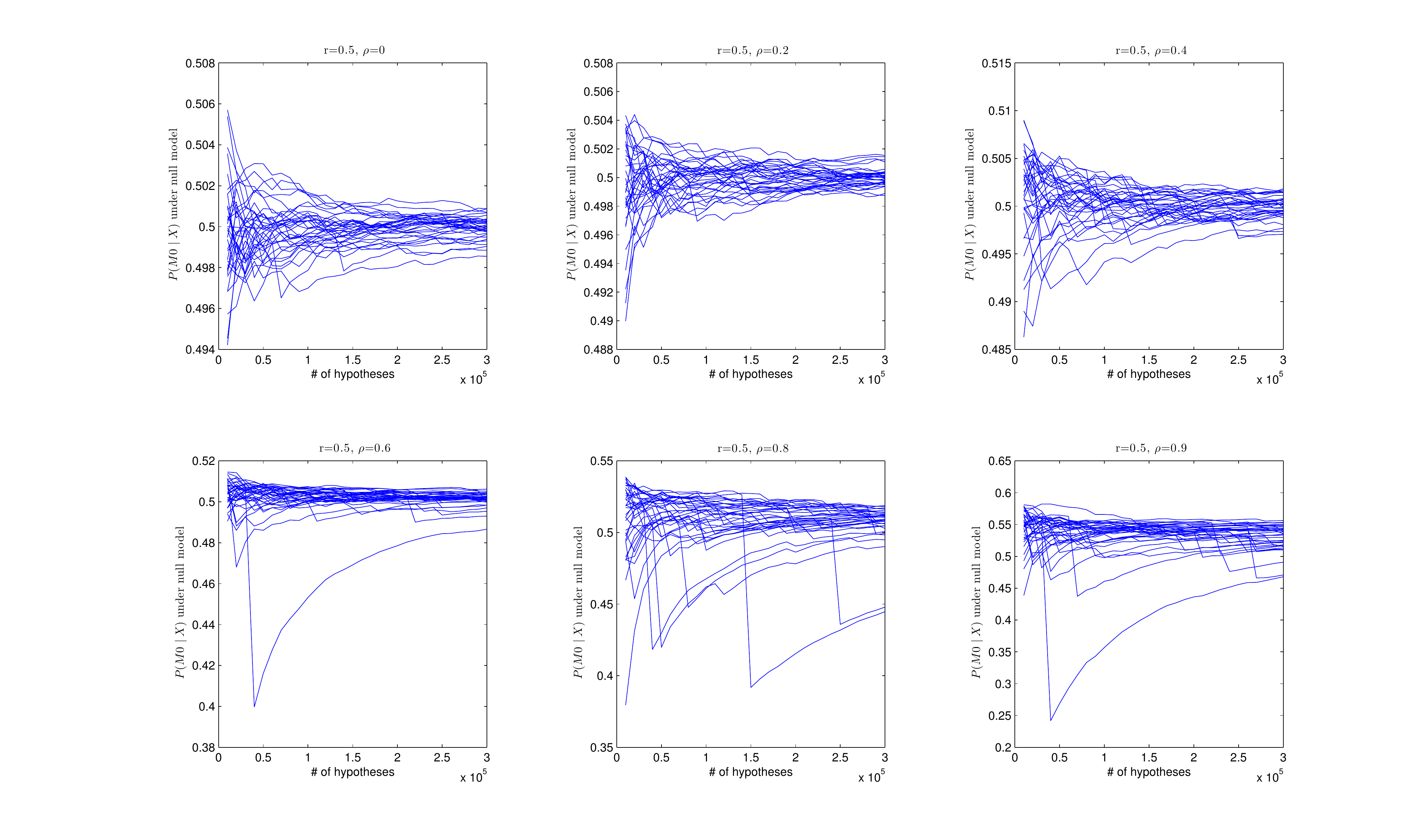}}
		\caption{Convergence of $P(M_0 \mid \boldsymbol{x})$ to the prior probability (0.5) under the null model. Each subplot has a different correlation and contains $50$ simulations. 
		}
		\label{fig:P(M0|X)=r}
	\end{figure}


	\subsection{False positive probability} \label{section: FPP}
	
	Here we focus on the major goal, to find the frequentist false positive probability under the null model of the Bayesian procedure. To begin, we must
	formally define the Bayesian procedure for detecting a signal.

	\begin{defn}[Bayesian detection criterion] \label{defn:model selection}
		Accept model $M_i$ if its posterior probability $P(M_i \mid \boldsymbol{x})$, is greater than a specified threshold $p\in(0,1)$. If multiple models pass this threshold, choose the one with largest posterior probability.
	\end{defn}

	\begin{defn}[False positive probability, FPP] Under the null model, the FPP is the frequentist probability of accepting a non-null model.
	\end{defn}

	\begin{thm}[False positive probability] \label{FPP}
		Under the null model, as $n \rightarrow \infty$,
		\begin{equation*}  \begin{split}
		P(false \,  \, positive \mid  r,\rho,\tau^2)
		%
		%
		&= O( n^{-\frac{1-\rho}{\tau^2}} (\log \,n)^{-1/2}) \,.
		\end{split} \end{equation*}
	\end{thm}

	\begin{proof} Under the null model, by (\ref{post_n_infinity}),
		$P(M_i \mid   \boldsymbol{x}) \geq p $ is equivalent to
		\begin{equation} \label{fdr_z_range}
		\begin{split}
		z_i^2 &\geq 2\bigg( \frac{1-\rho+\tau^2}{\tau^2} \bigg) ln\bigg( \frac{n}{1-r} \frac{p}{1-p} \sqrt{\frac{1-\rho+\tau^2}{1-\rho } }(1+o(1))   \bigg)  \\
		%
		%
		\end{split} \end{equation}

		By Fact \ref{normal_tail}: \small
		\begin{equation*} \begin{split}
		&\mathbbm{P}\bigg( |Z_i| \geq
		\underbrace{\sqrt{2\bigg( \frac{1-\rho+\tau^2}{\tau^2} \bigg) ln\bigg( \frac{n}{1-r}\frac{p}{1-p} \sqrt{\frac{1-\rho+\tau^2}{1-\rho }}\bigg)  +o(1) }} _{\gamma_n}  \bigg)  \\ &=
		1/n \underbrace{\left(
			\frac{\frac{2}{\sqrt{2\pi}} n^{-\frac{1-\rho}{\tau^2}}\bigg( \frac{1}{1-r}\frac{p}{1-p} \sqrt{\frac{1-\rho+\tau^2}{1-\rho}} \bigg)^{-(1+\frac{1-\rho}{\tau^2})}(1+o(1)) }
			{\sqrt{2 \bigg(\frac{1-\rho+\tau^2}{\tau^2} \bigg) ln\bigg(\frac{n}{1-r}\frac{p}{1-p}\sqrt{\frac{1-\rho+\tau^2}{1-\rho}})  \bigg)+o(1) } } \right)}_\textrm{$d_n$} +O\bigg(\frac{1}{n \, (\,\log\,n )^2}\bigg)    \,.
		\end{split} \end{equation*}
		\normalsize
		
		\begin{align*}
					&\mathbbm{P}(\text{any false positive} \mid M_0) =
		1-\prod^n_i \mathbbm{P}(|Z_i|<\gamma_n )\\
		&=1-(1-\mathbbm{P}(|Z_1| \geq \gamma_n))^n
		=1-\bigg(1-\frac{d_n}{n}\bigg)^n
		=1-(1-d_n)+O(d_n^2)\\
		&=
		\bigg( \frac{1}{n^{\frac{1-\rho}{\tau^2}} }  \bigg)
		\bigg(\frac{|\tau|  }{\sqrt{\pi}    (1-\rho+\tau^2)^{1+\frac{1-\rho}{2\tau^2}}} \bigg) %
		\bigg( \frac{(1-r)(1-p)}{p} \bigg) ^{1+\frac{1-\rho}{\tau^2}} \\
		&\qquad \bigg(\log \frac{n}{1-r}+ \log \bigg( \frac{p}{1-p} \sqrt{\frac{1-\rho+\tau^2}{1-\rho}}\bigg) \bigg)^{\frac{-1}{2}}(1+o(1)) \\
		&=O(n^{-(\frac{1-\rho}{\tau^2})}(\log \,n)^{\frac{-1}{2}})\,.
		\end{align*}
	\end{proof}
	
	
	
	This is a surprising and unsettling result: a standard Bayesian procedure yields a false positive probability that goes to zero at a polynomial rate.
	This is much too strong error probability control from a frequentist perspective; that it happens on the Bayesian side is surely indication that assuming we know $\tau^2$ is too strong an assumption. Hence we turn to a more flexible approach in the next section.
	
	\section{Adaptive choice of $\tau^2$}

To increase the frequentist power of the Bayes test, we consider adaptive choices of $\tau^2$. First, we consider the choice that maximizes the false positive probability. Then we consider a Type II maximum likelihood approach based on estimating $\tau^2$.

\subsection{The adaptive $\tau^2$ which maximizes FPP}
\label{sec.adaptiveFPP}

	\begin{thm}  \label{EB_likelihood}
		Given null model prior probability $r$, correlation $\rho$, and decision threshold $p$, as $n \rightarrow \infty$, the choice of $\tau^2$
that maximizes FPP is
\begin{equation}
\label{eq.maximizingtau}
{\tau}_n^2 = (1-\rho)[2\log n + \log\log n + 2\log \frac{p}{(1-p)(1-r)} +\log 2 ]\,.
\end{equation}	
The resulting FPP is
		\begin{equation}
\label{eq.adaptiveFPP}
 \begin{split}
		&\mathbbm{P}(\text{ false positive} \mid  null\, model\, ,\tau_n^2) \\
		&=
				\bigg(e^{-1/2}
				\sqrt{\frac{2}{\pi}} \bigg)		%
				\bigg(\frac{(1-p)(1-r)}{p}\bigg)
				\bigg(2\log\,n+\log\log\,n +c_\tau\bigg)^{-1} (1+o(1))	\,,	%
		\end{split} \end{equation}
		where $c_\tau = 2\log \frac{p}{(1-p)(1-r)} +\log 2+1\,.$
	\end{thm}
	\begin{proof}
		Without loss of generality, assume $\max \limits_{i} z_i^2=z_1^2$.
		By the model selection criteria  $(\ref{fdr_z_range})$, $z_1$ is a false positive if:
		\begin{equation}\label{FP_condition}
		\begin{split}
		z_1^2 &\geq 2\bigg(1+\frac{1-\rho}{\tau^2_n} \bigg) \log \bigg(
		\frac{n\,p}{(1-p)(1-r)}
		\sqrt{\frac{1-\rho+\tau^2_n}{1-\rho} } \bigg)+o(1)\\
		%
		%
		\end{split} \,.
		\end{equation}
		Lemma \ref{lem_max_tausq_FPP} establishes that (\ref{eq.maximizingtau}) maximizes the FPP and,
		with this choice of $\tau^2_n$, the rejection region becomes
		\begin{equation}\label{EB_z1}
				\begin{split}
				z_1^2 &> 2\, \log n  + \log \log\, n + \underbrace{2\log \frac{p}{(1-p)(1-r)}+ 1 + \log 2}_{c(p,r)}+ o(1)\,.
				%
				\end{split}
		\end{equation}
		Finally,	\small
			\begin{align*}
				& P( \text{ false positive} \mid \hat{\tau}^2_n, p,r)  \\
				&= 1-\left\{ 1-2 \left\{ \begin{array}{l}
				\frac{\frac{1}{\sqrt{2\pi}} \exp\left\{ -\frac{1}{2} 2\big(1+\frac{1-\rho}{\hat{\tau}_n^2} \big) \log \bigg(\frac{n}{1-r} \frac{p}{1-p}
					\sqrt{\frac{1-\rho+\hat{\tau}_n^2}{1-\rho}} \bigg) +o(1)\right\} }%
				{\sqrt{2\big(1+\frac{1-\rho}{\hat{\tau}_n^2} \big)
						log \big(\frac{np}{(1-r)(1-p)}\sqrt{\frac{1-\rho+\hat{\tau}_n^2} {1-\rho}} \big)+o(1)
					}} \\ +o\bigg(\frac{1}{n \, (\,\log\,n)^2}\bigg)\end{array} \right\} \right\}^n
					\\
				&=1- \left\{1-  \sqrt{\frac{2}{\pi}}\frac{\bigg( \frac{np}{(1-p)(1-r)} \sqrt{1+c_{n,\tau}} \bigg)^
						{-\big( 1+c_{n,\tau}^{-1}\big)} (1+o(1)) }
					{\sqrt{ \big(1+c_{n,\tau}^{-1} \big)
							\,2 \log\bigg( \frac{np}{(1-r)(1-p)} \sqrt{1+c_{n,\tau}}\bigg)   } }  \right\}^n\\
					&\hspace{ 2.3 in } \text{ where } c_{n,\tau} =2\log\,n+\log\log\,n+ c_\tau\\
					&=1- \left\{1- \frac{1}{n}
					\underbrace{\frac{
							\sqrt{2}
							\frac{(1-p)(1-r)}{p} \bigg(\frac{n\, p}{(1-p)(1-r)}
							\sqrt{1+c_{n,\tau}}
							\bigg)^{-c_{n,\tau}^{-1} }}
						{
							\sqrt{\pi \bigg(1+c_{n,\tau}\bigg)
								\bigg(
									2\log n + 2c + \log\log\,n +\log 2 +1
								\bigg)
							}}  }_\textrm{$d_n$} (1+o(1))\right\}^n\\
							&=d_n(1+o(1))+O\bigg(\frac{1}{(\log \,n+\log\log \,n)^2}\bigg) \\
							&=\bigg(e^{-1/2}
							\sqrt{\frac{2}{\pi}} \bigg)		%
							\bigg(\frac{(1-p)(1-r)}{p}\bigg)
							\bigg(2\log\,n+\log\log\,n +c_\tau\bigg)^{-1} (1+o(1))\,.\\
					\end{align*}
		\end{proof}
		\normalsize
	So, with this adaptive choice of $\tau^2$, the FPP only goes to zero at a $	1/(\log n+ \log\log n)$ rate, much slower than the polynomial rate achieved for fixed $\tau^2$.

	\begin{rem}
		Figure~\ref{fig:FPP_conv_diff_rho_adapted}  provides the simulated (red curve) and theoretical (in blue) false positive probability (FPP) with respect to the number of hypotheses (denoted by n). As expected, the simulated results match the theoretical prediction, the rate of convergence being around $1/{(2 \log n + \log\log n)}$.
		Note that the FPP does not become extremely small even for very large $n$.
		
	\end{rem}
	
	\begin{figure}[H]
		\centering
		\includegraphics[width=10cm, height=7cm]{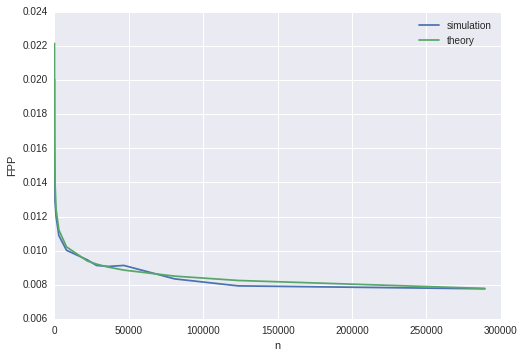}  
		\caption{\small Comparison of the simulated FPP and its asymptotic approximation when $p=r=0.5$, $\rho=0$ as $n$ varies from $10^1$ to $3e^{5}$, $\tau^2$ is the adaptive choice.}
		
		\label{fig:FPP_conv_diff_rho_adapted}
	\end{figure}

\subsection{Type II maximum likelihood estimation of $\tau^2$}
	
	The type II maximum likelihood approach to choice of the prior under the alternative model replaces a pre-specified $\tau^2$ with that prior variance, ${\hat \tau}_n^2$,  which maximizes the marginal likelihood over all possible $\tau^2$; see \cite{berger1985statistical} for discussion of this approach.

	\begin{lem}  \label{EB_lem}
		Let $\tilde{L}_n(\tau^2)$ be the marginal likelihood of $\tau^2$ given $(x_1,...,x_n)$, namely
		\[\tilde{L}_n(\tau^2) =\sum \limits_{i=0}^n P(M_i)m_i(\boldsymbol{x} \mid \tau^2) \,.\]
Defining
\[L_n(\tau^2) = \frac{1}{n\sqrt{1+\tau^2 a}} \sum_{i=1}^n \exp \left\{\frac{\tau^2}{2(1+\tau^2 a)} \big( \frac{x_i}{1-\rho}+bn\boldsymbol{\bar{x}} \big)^2 \right\} \,,\]
		the Type II mle, ${\hat \tau}_n^2$, can be found as
		\[  \argmax \limits _{\tau^2 } \tilde{L}_n(\tau^2 ) =\argmax \limits_{\tau^2} L_n(\tau^2) \,.\]
	\end{lem}
	
	\begin{proof}
		\small
		\begin{equation*}
		\begin{split}
		\tilde{L}_n
		&=r  |\Sigma_0|^{\frac{-1}{2}} \exp \{ \frac{-1}{2}\boldsymbol{x}'\Sigma_0^{-1} \boldsymbol{x}\} +
		\frac{1-r}{n}  |\Sigma_1|^{\frac{-1}{2}}  \sum\limits ^n_{i=1}  \exp \{ \frac{-1}{2}\boldsymbol{x}'\Sigma_i^{-1} \boldsymbol{x}\} \\
		&= r  |\Sigma_0|^{\frac{-1}{2}} \exp \left\{ \frac{-1}{2}\boldsymbol{x}'\Sigma_0^{-1} \boldsymbol{x}\right\}
		+\frac{1-r}{n}  |\Sigma_0|^{\frac{-1}{2}} (1+\tau^2 a)^{\frac{-1}{2}} \\
		&\hspace{2cm}*\exp\left\{\frac{-1}{2} \boldsymbol{x}' \Sigma_0^{-1} \boldsymbol{x}\right\}
		\sum \limits^n_{i=1} 		\exp \left\{\frac{\tau^2}{2(1+\tau^2a)}
		\bigg(x_i(a-b)+bn\boldsymbol{\bar{x}} \bigg)^2 \right\}  \\
		&=|\Sigma_0|^{\frac{-1}{2}} \exp \left\{ \frac{-1}{2}\boldsymbol{x}'\Sigma_0^{-1} \boldsymbol{x}\right\}
		*\left\{ \begin{array}{l}
		r+\frac{1-r}{n} (1+\tau^2 a)^{\frac{-1}{2}}\\
		*\sum \limits^n_{i=1} \exp \left\{\frac{\tau^2}{2(1+\tau^2a)}
		\bigg(x_i(a-b)+bn\boldsymbol{\bar{x}} \bigg)^2 \right\}
		\end{array}\right\}\,.
		\end{split}
		\normalsize
		\end{equation*}
				Noting that $a$, $b$, $\Sigma_0$ and $\boldsymbol{x}'\Sigma_0^{-1}\boldsymbol{x}$ are independent of $\tau^2$,  the result follows.
	\end{proof}
	
	



\begin{thm} [Type II MLE false positive probability] \label{EB_likelihood}
		Given null prior probability $r$, correlation $\rho$, and decision threshold $p$, as $n \rightarrow \infty$
		\begin{equation*}  \begin{split}
		\mathbbm{P}(false \,  \, positive \mid  null\, model\, ,{\hat \tau}_n^2) =\frac{1}{log\, n} \bigg(\frac{1}{k^*}-\frac{1}{2} \bigg)
		(1+o(1))\,,
		\end{split} \end{equation*}
		where $k^*$ satisfies:
		\begin{align} \label{k*}
		-2\, log\, \bigg( \sqrt{\pi} \bigg( \frac{1}{k^*} -\frac{1}{2}\bigg)\bigg)
		&= log \, k^*  +2\, log \bigg( \frac{p}{(1-p)(1-r)}\bigg) +2\bigg( \frac{1}{k^*}\bigg)\,.
		\end{align}
		
\end{thm}

\begin{proof}		First, Lemma \ref{lem_max_tausq_FPP} shows that (\ref{EB_z1}) provides the absolute lower bound for $z_1^2$ to be in the rejection region; namely  $2 \log n +\log \log n +c(p,r)$. So the rejection region, denote it by $\Omega$, corresponding to the Type II MLE choice of $\tau^2$ must
 be a subset of $(2 \log n +\log \log n + c(p,r), \infty)$. Divide this interval into
 \begin{eqnarray*}
 \Omega_1 &=& (2 \log n +\log \log n +c(p,r), 2 \log n +\log \log n +c(p,r)+ K) \\
  \Omega_2 &=& (2 \log n +\log \log n +c(p,r)+K, \infty) \,,
  \end{eqnarray*}
  where $K$ will be chosen large, but fixed. We first determine $\Omega \cap \Omega_1$.
 
 For any $z_1^2 = 2 \log n +\log \log n + c \in \Omega_1$,  Lemma \ref{EB_max}, shows that the Type-II MLE estimate is			
			\[\hat{\tau}^2_n=(1-\rho)k(c) \, log\, n (1+o(1)) \text{ where }   							
			k(c)=\big( 1/2+ exp\left\{ -c/2\right\}/\sqrt{\pi} \big)^{-1}\,. \] 
Thus, letting $z_1^{*2} = 2 \log n +\log \log n + c^* $ denote the smallest value in $\Omega \cap \Omega_1$ (if it exists) and letting
$\hat{\tau}_n^{*2}=(1-\rho)k(c^*) \, log\, n (1+o(1))$ denote the corresponding Type-II MLE estimate, the smallest value must satisfy,
	by (\ref{FP_condition}), 
	\small
		%
		\begin{equation*}
		\begin{split}
		z_1^{*2} &= 2\bigg(1+\frac{1-\rho}{\hat{\tau}_n^{*2}} \bigg) ln \bigg(\frac{n}{1-r}\frac{p}{1-p}\sqrt{\frac{1-\rho+\hat{\tau}_n^{*2}}{1-\rho} } \bigg)+o(1) \\
		&=2\bigg(1+\frac{1}{k(c^*) \, log\, n (1+o(1))} \bigg)
		\\&\hspace{1cm}*\bigg(log \, n + log \frac{p}{(1-p)(1-r)}+\frac{1}{2} log(1+k(c^*) \,log\, n\,(1+o(1)))\bigg)+o(1)\\
		&= 2\, log \, n +log\, log \,n + \underbrace{\bigg[ log\, k(c^*)+ 2\, log \bigg(\frac{p}{(1-p)(1-r)}\bigg)+2\bigg(\frac{1}{k(c^*)}\bigg)  \bigg]}_{l^*}+o(1)\bigg) \,.
		\end{split}
		\end{equation*}
		This is equivalent to 
			\begin{equation} \label{c*}
		\begin{split}
		c^*&=log\, k(c^*)+ 2\, log \bigg(\frac{p}{(1-p)(1-r)}\bigg)+2\bigg(\frac{1}{k(c^*)}\bigg) \\
		&=-log \bigg( \frac{1}{2}+\frac{\exp\left\{-c^*/2\right\}}{\sqrt{\pi}} \bigg) +2\,log\,\bigg(\frac{p}{(1-p)(1-r)}\bigg) +1+2\frac{ \exp\left\{-c^*/2\right\}}{\sqrt{\pi}}\,,
		\end{split}
		\end{equation}	
		\normalsize
	which, using Lemma \ref{l>0} (which shows that $l^* > c(p,r)$), can easily be shown to have a unique solution in $\Omega \cap \Omega_1$ (assuming K is larger than, say, $4\,log\,(\frac{p}{(1-p)(1-r)})$).
It also also then easy to show that 
$$ \Omega \cap \Omega_1 = (2\, log \, n +log\, log \,n + l^* +o(1), 2\, log \, n +log\, log \,n + c(p,r) +K) \,.$$

	By \ref{normal_tail},
	\begin{equation*}
	\begin{split}
		&\frac{P(\Omega_2)}{P(\Omega \cap \Omega_1)} \\
& \leq
		\frac{\exp{(-[2\log n + \log\log n + c(p,r)+K]/2)}/\sqrt{2\log n + \log\log n +c(p,r)+K}}
		{\frac{\sqrt{2\log n + \log\log n +l^*}\exp{(-[2\log n + \log\log n + l^*]/2)}}{2\log n + \log\log n +c(p,r)+l^*}- \frac{\exp{(-[2\log n + \log\log n + c(p,r)+K]/2)}}{\sqrt{2\log n + \log\log n +c(p,r)+K}}} \\
 &= \left(\exp{([c(p,r)+K-l^*]/2)}-1\right)^{-1} (1+o(1)) \,.
	\end{split}
	\end{equation*}
 $c(p,r)$ and $l^*$ are fixed, we can clearly choose $K$ large enough to make this smaller than any specified $\epsilon$. Hence the
region $\Omega_2$ can be ignored in the computation of the FPP. (It is almost certainly part of the rejection region, but we do not know
what $\hat{\tau}_n^{2}$ is for observations in that region and, hence can't say for sure.)

We can also use the same argument to say that
$$ P(\Omega \cap \Omega_1) = P((2\log n + \log\log n + l^*, \infty))(1+\epsilon) \,.$$
Writing $k^*$ for $k(c^*)$, it follows that the FPP is
		\begin{equation*}
		\begin{split}
          &FPP = 1- \left\{1-  \frac{\frac{2}{\sqrt{2\pi}}
				\exp \left\{\frac{-1}{2}(
				2\log n + \log \log n + l^* )
									\right\}(1+o(1))(1+\epsilon) }
						{\sqrt{
					2\log n + \log \log n + l^*  } }   \right\}^n\\
			&=1- \left\{1-\frac{1}{n}
						\sqrt{\frac{2}{\pi}} \frac{
								(1-p)(1-r)\exp(-1/k^*)  }
								{p\sqrt{k^*} \sqrt{(\log n)(2\log n + \log \log n)} }  (1+o(1))(1+\epsilon) \right\}^n\\
			&=\sqrt{\frac{2 }{\pi k^*}  } \exp\left\{\frac{-1}{k^*} \right\}
			\bigg(\frac{(1-p)(1-r)}{p}\bigg)  \bigg[(\log n )(2\log n + \log \log n)\bigg]^{-1/2}  \\
							&\hspace{10 cm}(1+o(1))(1+\epsilon)\\
			&=\bigg(\frac{1}{k^*}-\frac{1}{2}\bigg)
		  \frac{1}{log\,n} (1+o(1))(1+\epsilon)
			\quad \text{by (\ref{c*})}\,. 
		\end{split}
			\end{equation*}
Since $\epsilon$ can be made arbitrarily small, the result follows.
		\end{proof}	
	
\normalsize

Note that (\ref{k*}) can be solved numerically. For instance, when $p=r=0.5$, $k^*\approx 1.6142$. The solution of $\frac{1}{k^*}-\frac{1}{2}$ with respect to $\frac{p}{(1-r)(1-p)}$ is, indeed, given in Figure \ref{fig:num_kstar}.

The Type II MLE FPP converges to 0 at a logarithmic rate in n, as did the maximal Bayesian FPP. Thus both are far less conservative than the Bayesian procedures with specified $\tau^2$. Finally, it is interesting that neither of the adaptive asymptotic FPP's depend on $\rho$.



\begin{rem}
	Figure \ref{fig:FPP_fixed} demonstrates how the threshold $p$ (Definition \ref{defn:model selection}) can be chosen to achieve
a fixed FPP of 0.05. Because, for a fixed $p$, the FPP goes to zero as a function of $n$, smaller $p$ are needed to achieve a fixed FPP as $n$ grows.
Note that the variation in $p$ is actually quite small over the very large range of $n$ considered in the figure.	
\end{rem}

\begin{figure}[H]
	\includegraphics[scale=0.3]{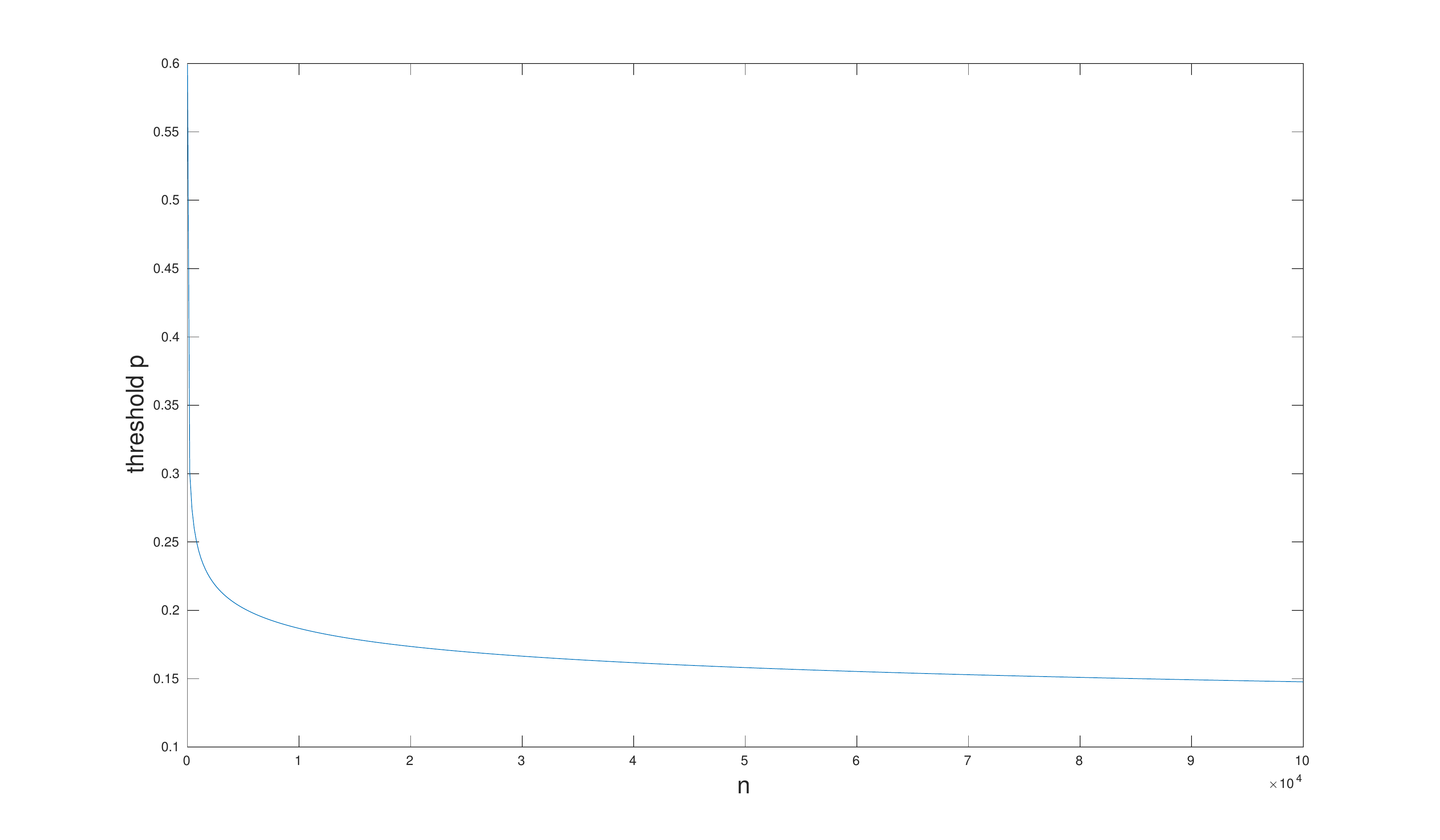}
	\caption{For fixed prior probability of
0.5 for the null model, this gives, as the number of hypotheses $n$ increases, the Bayesian
 threshold probability $p$ that would achieve an FPP of $=0.05$. }
	\label{fig:FPP_fixed}
\end{figure}

\begin{rem}
	Figure \ref{fig:num_kstar} gives the value of $\frac{1}{k^*}-\frac{1}{2}$ for different $\frac{p}{(1-r)(1-p)}$.  \end{rem}

\begin{figure}[H]
	\includegraphics[scale=0.5]{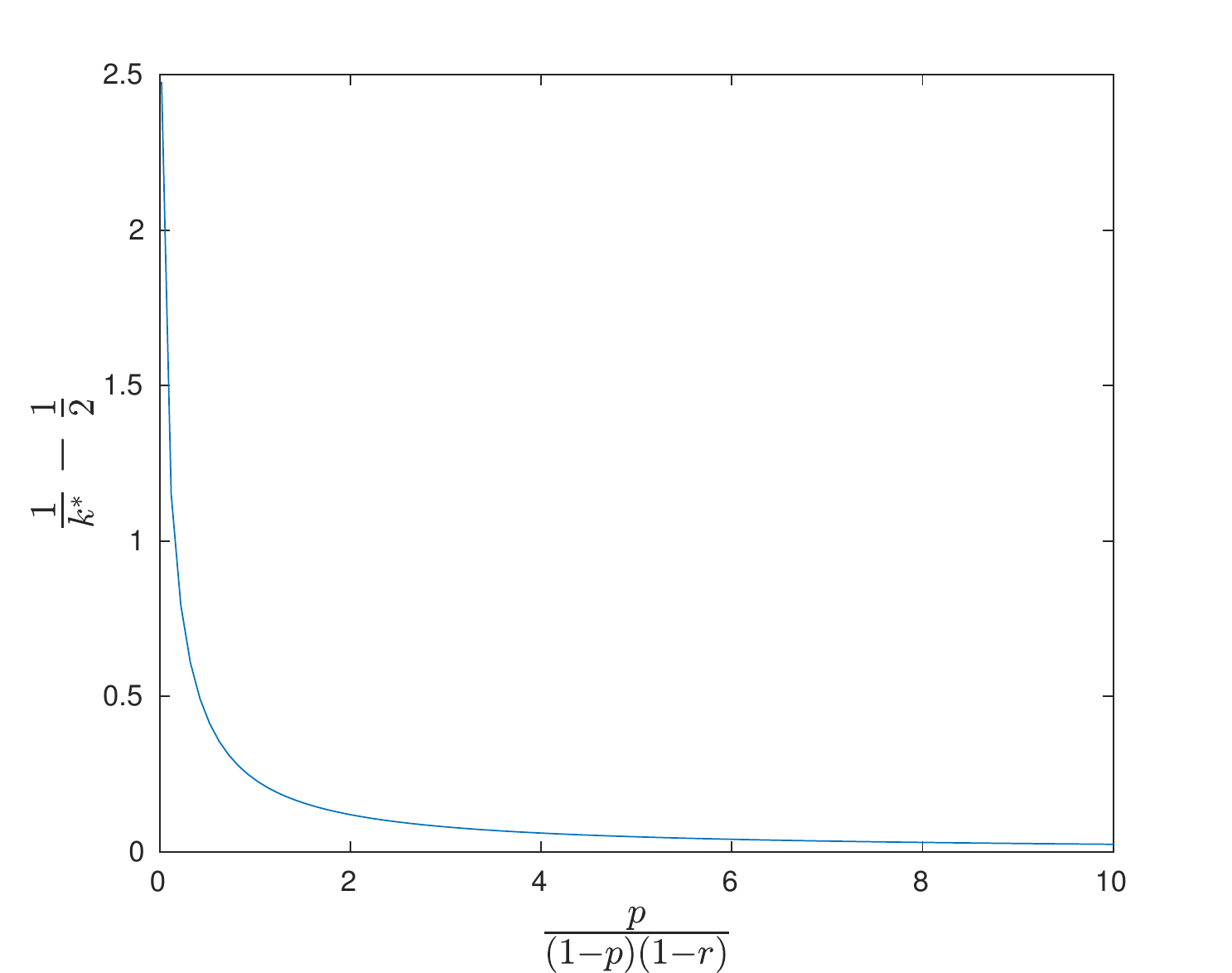}
	\caption{Solution of $\frac{1}{k^*}-\frac{1}{2}$ (y-axis) with respect to different $\frac{p}{(1-p)(1-r)}$ (x-axis).}
	\label{fig:num_kstar}
\end{figure}

	\begin{rem}
		Figure \ref{fig:power_analysis_power_versus_theta} demonstrates the how the detection power varies when the signal size increases.
	\end{rem}

\begin{figure}[H]
	\includegraphics[scale=0.5]{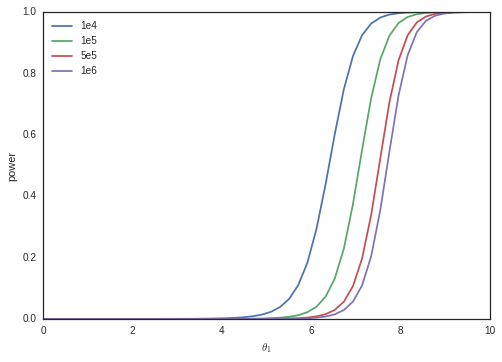}
	\caption{Power versus $\theta_i$ for fixed $p=r=0.5$, $\rho=0$, and different $n$ (see the color legend on the top left). Each point is the average acceptance rate of the true non-null model when $\theta_i$ is as specified on the x-axis. }
	\label{fig:power_analysis_power_versus_theta}
\end{figure}
				\section{Analysis as the information grows }
				In this section, we generalize model (\ref{model}) to the scenario where each channel has $m$ i.i.d observations. Then the sample mean satisfies
				\begin{equation}
				{\boldsymbol{\bar{X}}}\sim multinorm \left(
				\begin{pmatrix}
				\theta_1  \\
				\theta_2  \\
				\vdots  \\
				\theta_n
				\end{pmatrix}
				,
				\frac{1}{m}
				\begin{pmatrix}
				1 & \rho & \cdots & \rho \\
				\rho & 1 & \cdots & \rho \\
				\vdots  & \vdots  & \ddots & \vdots  \\
				\rho & \rho & \cdots & 1
				\end{pmatrix}  \right)\,.
				\end{equation}
				Hence, $m$ can be seen as the precision of $\boldsymbol{\bar{X}}$. More generally, we will replace
				$1/m$ by a function $\sigma^2_n$, where $\sigma_n^2$ decreases to zero as $n$ grows.

The theorem below gives the rate of decrease of $\sigma^2_n$
				which guarantees consistency. For the i.i.d. case, consistency of all models is only guaranteed if $m$ grows faster than $\log n$; consistency fails if $m$ grows slower than $\log n$; and consistency depends on the parameter value if $m$ is $O(\log n)$.

				\begin{thm} \label{information_grows_model}
					Consider model (\ref{model}), with the altered covariance matrix below:
					\begin{equation}\label{information_grows_model}
					\boldsymbol{X}\sim multinorm \left(
					\begin{pmatrix}
					\theta_1  \\
					\theta_2  \\
					\vdots  \\
					\theta_n
					\end{pmatrix}
					,
					\sigma_n^2
					\begin{pmatrix}
					1 & \rho & \cdots & \rho \\
					\rho & 1 & \cdots & \rho \\
					\vdots  & \vdots  & \ddots & \vdots  \\
					\rho & \rho & \cdots & 1
					\end{pmatrix}  \right)\,,
					\end{equation}
					
					  \begin{enumerate}
					\item  When $\sigma_n^2 \log\,n  \rightarrow 0$, consistency holds for both the null and alternative models.
					\item  When $\sigma_n^2 \log\,n \rightarrow d \in (0,\infty)$,
					\begin{itemize}
					\item\quad  Under $M_0$:
					$P(M_0\mid \boldsymbol{X}) \rightarrow \big(1+\frac{1-r}{r} \big[2\Phi\big(\frac{(1-\rho)d}{\tau^2}\big)-1\big]\big)^{-1}\,,$ failing to be consistent.
										\item\quad  Under an alternative model $M_j$,				
					if d $\in (0,\frac{\theta_j^2}{2(1-\rho)})$, consistency holds for $M_j$, whereas
					consistency does not hold otherwise.
					\end{itemize}							
					\item  When $\sigma_n^2 \log\,n\rightarrow \infty$ and $\sigma_n^2 \log\,n = o(\log\,n)$, consistency does not hold for any model. In addition, when the null hypothesis is true,
					\[P(M_0\mid \boldsymbol{X})\rightarrow P(M_0) \,.\]
					 \end{enumerate}
				\end{thm}

				\begin{proof}
					Write $\sigma_n^2 =  d_n/log\,n $,
					$X_i^* = X_i/\sigma_n$, and $\theta^*_i =\theta_i/\sigma_n$. Then a nonzero $\theta^*_i$ has prior $N(0, \tau^2/\sigma_n^2)$ and (\ref{information_grows_model}) becomes:
					\[
					{\boldsymbol{X}^*}\sim multinorm \left(
					\begin{pmatrix}
					\theta_1^*   \\
					\theta_2^*  \\
					\vdots  \\
					\theta_n^*
					\end{pmatrix}
					,
					\begin{pmatrix}
					1 & \rho & \cdots & \rho \\
					\rho & 1 & \cdots & \rho \\
					\vdots  & \vdots  & \ddots & \vdots  \\
					\rho & \rho & \cdots & 1
					\end{pmatrix}  \right)\,.
					\]
					
					\small					
					\textit{Under the null model}: by Theorem \ref{post_full},
					\begin{equation} \label{information_grows_M0}
					\begin{split}
					&P(M_0 \mid \boldsymbol{x})^{-1} \\
					&= 1+ \bigg( \frac{1-r}{n\,r}\bigg)
					\sqrt{\frac{\sigma_n^2}{\sigma_n^2+\tau^2 a_n}}* \\
					&\qquad \qquad \sum \limits^n_{i=1} \exp\left\{ \frac{\tau^2}{2(\sigma_n^2+\tau^2a_n)}
					\bigg( \frac{z_i}{\sqrt{1-\rho}} + \frac{\sqrt{\rho}z}{1-\rho}
					+b_n n\sqrt{1-\rho}\boldsymbol{\bar{z}} +b_ nn \sqrt{\rho}z \bigg)^2\right\}  \\
					&=  1+ \bigg(\frac{1-r}{n r}\bigg) \sqrt{ \frac{\sigma^2_n}{\sigma^2_n +\tau^2/(1-\rho)}}
					\sum \limits^n_{i=1} \exp\left\{ \frac{\tau^2}{2(\sigma_n^2+\tau^2/(1-\rho))}
					\frac{z_i^2}{(1-\rho)}
					\right\} \bigg( 1+ O(\frac{1}{\sqrt{n}}) \bigg)   \\
					&=  1+ \big(\frac{1-r}{n r}\big) \sqrt{ 1-c_n}
					\sum \limits^n_{i=1} \exp\left\{ \frac{c_n}{2} z_i^2 \right\} \big( 1+ O(1/\sqrt{n}) \big)\,,
					\end{split}
					\end{equation}
					\normalsize
					where
					\[\begin{cases}
					a_n =\frac{1}{1-\rho} +\frac{-\rho}{(1-\rho)(1+(n-1)\rho)} =\frac{1}{1-\rho} +O(1/n)\,, \\
					n b_n =\frac{-1}{1-\rho} +\frac{1-\rho}{(1-\rho)(1+(n-1)\rho)} =\frac{-1}{1-\rho} +O(1/n)\,,\\
					c_n= \frac{\tau^2}{(1-\rho)\sigma^2_n +\tau^2} = \frac{\tau^2}{(1-\rho)d_n/log\,n+\tau^2}\,.
					\end{cases}\]
					
					

					Since $d_n = o(\log\,n)$, $1-c_n= o(1)$, and $0<c_n< 1$, one can apply  Theorem \ref{EB_limit_general} to get the asymptotic analysis of
					(\ref{information_grows_M0}):
					\begin{equation*} 
					\begin{split}
					& 2 \Phi \bigg( \frac{2(1-c_n)}{c_n} \log\, \frac{n}{\sqrt{1-c_n}}  \bigg) -1   
					= \frac{(1-\rho)\sigma^2_n}{\tau^2} \log \bigg( n \sqrt{\frac{(1-\rho)\sigma^2_n +\tau^2 }{(1-\rho)\sigma^2_n}} \bigg) \\
					 &= \frac{(1-\rho)d_n }{\tau^2 \log \, n} \bigg[ \log \bigg( n\sqrt{\frac{log\,n}{d_n}}\bigg) +O(1)   \bigg] \\
					&=  \frac{1-\rho}{\tau^2} \bigg[ d_n + \frac{1}{2}\bigg(\frac{d_n}{log\,n}\bigg) \log\bigg(\frac{log\,n}{d_n}\bigg)\bigg] (1+O(1))
					\rightarrow
					\begin{cases}
					\infty                            &\mbox{ if } d_n \rightarrow \infty\,, \\
					\frac{1-\rho}{\tau^2}d            &\mbox{ if } d_n \rightarrow d\,,\\
					0                                 &\mbox{ if } d_n \rightarrow 0\,.
					\end{cases}
					\end{split}
					\end{equation*}
					Hence, under the null hypothesis,
					\[P(M_0\mid \boldsymbol{X}) \rightarrow
					\begin{cases}
					P(M_0)                                                       &\mbox{ if } d_n \rightarrow \infty\,, \\
					\big(1+(\frac{1-r}{r})(2\Phi(\frac{1-\rho}{\tau^2}d)-1)\big)^{-1}     &\mbox{ if } d_n \rightarrow d\,, \\
					1                                                       &\mbox{ if } d_n \rightarrow 0\,.
					\end{cases}\]
					
					\textit{Under the alternative model $M_j$}: by Theorem \ref{post_full},
					\begin{equation*}
					\begin{split}
					&P(M_j \mid \boldsymbol{x})^{-1} \\&= \sqrt{1+\frac{a_n\tau^2}{\sigma^2_n}} \bigg(\frac{n\,r}{1-r}\bigg)
					\exp\left\{\frac{-\tau^2}{2(\sigma^2_n +a_n\tau^2)} \bigg[\frac{\theta_j^2}{\sigma^2_n(1-\rho)^2} + O(\frac{1}{\sigma_n}) \bigg] \right\}\\
					&+ 1 + \sum\limits_{k\neq j}^n \exp \left\{\frac{-\tau^2}{2(\sigma^2_n+a_n\tau^2)} \bigg[\frac{\theta_j^2}{2(1-\rho)^2\sigma_n^2}+O(\frac{1}{\sigma_n}) \bigg]\right\}\,.
					\end{split}
					\end{equation*}
					The first term is \small
					\begin{equation*}
					\begin{split}
					&\sqrt{1+\frac{\tau^2/(1-\rho)}{\sigma_n^2}}\bigg(\frac{n\,r}{1-r}\bigg)
					\exp\left\{ \frac{-\tau^2}{2(\sigma_n^2+\tau^2/(1-\rho))}\bigg[\frac{\theta_j^2}{(1-\rho)^2\sigma_n^2} \bigg] \right\}
					 \bigg(1+O\bigg(\sqrt{\frac{log\,n }{d_n}}\bigg)\bigg)\\
					&= \sqrt{\tau^2/(1-\rho)} \sqrt{\frac{log\,n}{d_n}}\bigg(\frac{r}{1-r}\bigg) n^{1-\frac{\theta_j^2}{2(1-\rho)d_n}}
					\bigg(1+O\bigg(\sqrt{\frac{log\,n}{d_n}}\bigg)\bigg)\\
					&\mbox{which $\rightarrow 0$ if and only if } \lim \limits_{n\rightarrow \infty} d_n< \frac{\theta_j^2}{2(1-\rho)}\,.
					\end{split}
					\normalsize
					\end{equation*}
										And the last term,
					\begin{equation*}
					\begin{split}
					&\sum \limits_{k\neq j}^n
					\exp\left\{ \frac{-\tau^2}{2(\sigma_n^2+\tau^2/(1-\rho))}\bigg[\frac{\theta_j^2}{(1-\rho)^2\sigma_n^2} \bigg) \bigg] \right\}\bigg(1+O\bigg(\sqrt{\frac{log\,n }{d_n}}\bigg)\bigg)\\
					&= n \exp \left\{-\frac{\theta_j^2}{2(1-\rho)}\frac{log\,n}{d_n} \right\}\bigg(1+O\bigg(\sqrt{\frac{log\,n}{d_n}}\bigg)\bigg) \\
					&= n^{1-\frac{\theta_j^2}{2(1-\rho)d_n}}\bigg(1+O\bigg(\sqrt{\frac{log\,n}{d_n}}\bigg)\bigg) \,, \\
					&\rightarrow 0 \text{ when } \lim \limits_{n\rightarrow \infty} d_n< \frac{\theta_j^2}{2(1-\rho)}\,.
					\end{split}
					\end{equation*}
					
				\end{proof}
				

				\section{Conclusions}
				
The main purpose of this work was to gain understanding of the behavior of Bayesian procedures that control for multiple testing, under a scenario of high dependence among test statistics, where frequentist methods for multiplicity control become more difficult to implement when trying to maintain high power. In Section \ref{sec.power}, the Bayesian procedure was shown to have unexpectedly high power as the correlation gets large, providing an illustration of the gains that can be had by approaching multiplicity control from the Bayesian side. (Bayes theorem often produces things that we could not have produced through our intuition alone.)

The other main issue concerning the behavior of the Bayesian procedure is the extent to which it also exhibits desirable frequentist control. Surprising to us was that the Bayesian procedure exhibited too-strong frequentist control, with the FPP (false positive probability under the null model) going to zero at a polynomial rate, as the number $n$ of tests grows. To a Bayesian who believed in the prior distribution that was utilized this would not be viewed as a problem, but we tend to prefer procedures that have a dual Bayesian/frequentist interpretation. To this end, adaptive versions of the Bayesian procedure were considered, and found to have FPP's going to 0 at the much slower $1/\log n$ rate; indeed, unless $n$ is huge, the resulting FPPs were reasonably moderate.

A number of other surprises were also encountered, such as the fact that, as the number of tests $n$ grows, the posterior probability of the null model converges to its prior probability. (This is actually a very general phenomenon that will be reported elsewhere.) The situation of having i.i.d. replicate observations $m$ was also considered, and it was shown that one needs $m$ to grow faster than $\log n$ to achieve consistency, under both the null and alternative models.

Methodologically, if a frequentist were to encounter this particular multiple testing problem and desired a procedure that is fully powered and achieves an FPP of $\alpha$, we would suggest using the adaptive Bayesian procedure in Section \ref{sec.adaptiveFPP}. One solves (\ref{eq.adaptiveFPP}) for the Bayesian rejection threshold $p$ (with, say, the default choice of $r=1/2$ for the prior probability of the null model), and then rejects the null and accepts $M_i$ if
$P(M_i \mid \boldsymbol{x}) > p$, where $P(M_i \mid \boldsymbol{x})$ is as in (\ref{eq:post_prob_simplfy}) with $\tau^2$ chosen as in (\ref{eq.maximizingtau}). This Bayesian procedure will have the unusual power benefits outlined in Section \ref{sec.power} when the correlation is high, while achieving the
desired frequentist FPP (at least asymptotically) and likely having the greatest power against alternatives, since the $\tau^2$ in (\ref{eq.maximizingtau}) was chosen, in essence, to maximize the power.

				\section{Appendix}
				
				\subsection{Normal Theory}
				
				\begin{lem} \label{xandz_lem}
					\begin{equation*}
					\boldsymbol{X}\sim multinorm \left(
					\begin{pmatrix}
					\theta_1   \\
					\theta_2  \\
					\vdots  \\
					\theta_n
					\end{pmatrix}
					,
					\begin{pmatrix}
					1 & \rho & \cdots & \rho \\
					\rho & 1 & \cdots & \rho \\
					\vdots  & \vdots  & \ddots & \vdots  \\
					\rho & \rho & \cdots & 1
					\end{pmatrix}  \right)
					\end{equation*} is equivalent to
					\begin{equation} \label{xandz}
					X_i =\theta_i+\sqrt{\rho}Z+\sqrt{1-\rho}Z_i \, \, \forall \, i \in \{1,2,...,n \}\,,
					\end{equation}
					where $Z,Z_1,...,Z_n \sim iid \, N(0,1)$.
					Furthermore, if $\theta_j =0 \, \, \, \forall j$, then, as $n\rightarrow \infty$,
					\begin{equation}\label{xandz2}
					\begin{cases}
					\frac{\boldsymbol{\bar{x}}}{\sqrt{\rho}}=z+O\bigg(\frac{1}{\sqrt{n}} \bigg)\\
					\frac{x_i-\boldsymbol{\bar{x}}}{\sqrt{1-\rho}} =z_i +O\bigg( \frac{1}{\sqrt{n}}\bigg)
					\end{cases} \,.
					\end{equation}
				\end{lem}
				\begin{proof}
					It is straightforward to show that the expectation and covariance of (\ref{xandz}) are as desired.
					(\ref{xandz2}) follows from the definitions and the central limit theorem.
					
				\end{proof}

				\begin{fact}[Normal tail probability] \label{normal_tail} Letting $\Phi(t)$ denote the cumulative  distribution function of the standard normal distribution,
					\begin{equation*}
					\frac{t\frac{1}{\sqrt{2\pi}} e^{-t^2/2}   }{t^2+1}\leq\ 1-\Phi(t)\leq \frac{\frac{1}{\sqrt{2\pi}} e^{-t^2/2}   }{t}
					\end{equation*}
					\begin{equation*}
					1-\Phi(t) =\frac{\frac{1}{\sqrt{2\pi}} e^{-t^2/2}}{t} +O\bigg(\frac{e^{-t^2/2}}{t^3} \bigg)
					\end{equation*}
				\end{fact}
				The proof can be found in \cite{durrett2010probability}.

				By expanding $a,b$ in Theorem \ref{eq:post_prob_simplfy}, one obtains the following explicit form for the posterior probabilities:
				\begin{cor}  The posterior of any non-null model $M_i$ is:
					\normalsize
					\begin{equation} 
					\begin{split}
					&P(M_i \mid  \boldsymbol{x})= \\
					&\left\{ \begin{array}{l}
					\bigg( \sqrt{ \frac{(1-\rho+\tau^2)(1+(n-1)\rho)-\tau^2\rho}{(1+(n-1)\rho)(1-\rho)} } \bigg(\frac{n \, r}{1-r}\bigg)*
					\\
					\exp \left\{ \frac{-\tau^2}{2}\frac{1+(n-1)\rho}{[(1-\rho+\tau^2)(1+(n-1)\rho)-\tau^2\rho](1-\rho)}  \left( (x_i-\boldsymbol{\bar{x}})+\frac{(1-\rho)\boldsymbol{\bar{x}}}{1+(n-1)\rho} \right)^2   \right\}  +\\
					\sum \limits_{k=1}^n \exp\bigg[ \frac{-\tau^2}{2} \frac{1+(n-1)\rho}{[(1-\rho+\tau^2)(1+(n-1)\rho)-\tau^2\rho](1-\rho)} \\
					\hspace{2 cm}\bigg((x_i+x_k-2\boldsymbol{\bar{x}})(x_i-x_k)+2\frac{\boldsymbol{\bar{x}}(x_i-x_k)(1-\rho)}{1+(n-1)\rho)} \bigg) \bigg]
					\end{array} \right\}^{-1} \,.
					\end{split}
					\end{equation}
					Alternatively, in terms of $z(\boldsymbol{x})$, with $z_i = z_i(\boldsymbol{x})$:
					\begin{equation} \label{eq:post_prob_full}
					\begin{split}
					\small
					&P(M_i \mid  \boldsymbol{x})= \\
					&
					\left\{ \begin{array}{l}
					\bigg( \sqrt{ \frac{(1-\rho+\tau^2)(1+(n-1)\rho)-\tau^2\rho}{(1+(n-1)\rho)(1-\rho)} } \bigg(\frac{n \, r}{1-r}\bigg)* \\
					\exp \bigg[ \frac{-\tau^2}{2}\frac{1+(n-1)\rho}{[(1-\rho+\tau^2)(1+(n-1)\rho)-\tau^2\rho](1-\rho)}\\
					\hspace{0.5cm} \bigg( \theta_i-\bar{\boldsymbol{\theta}}+\sqrt{1-\rho}(z_i-\boldsymbol{\bar{z}})+
					 \frac{1-\rho}{1+(n-1)\rho}
					(\bar{\boldsymbol{\theta}}+\sqrt{\rho}z+\sqrt{1-\rho} \boldsymbol{\bar{z}})   \bigg) ^2  \bigg]  \\
					+\sum \limits_{k=1}^n \exp \bigg[
					\frac{-\tau^2}{2}
					\frac{1+(n-1)\rho}{[(1-\rho+\tau^2)(1+(n-1)\rho)-\tau^2\rho]}\\
					\hspace{2cm} \bigg( \frac{[\theta_i+\theta_k-2\bar{\boldsymbol{\theta}}+\sqrt{1-\rho}(z_i+z_k-2\boldsymbol{\bar{z}})][\theta_i-\theta_k+\sqrt{1-\rho}(z_i-z_k)]}{1-\rho}+ \\
					\hspace{4cm} 2\frac{[\bar{\boldsymbol{\theta}}+\sqrt{\rho}z+\sqrt{1-\rho}\boldsymbol{\bar{z}}][\theta_i-\theta_k+\sqrt{1-\rho}(z_i-z_k)]}{1+(n-1)\rho}\bigg)
					\bigg]
					\end{array} \right\}^{-1}\,.
					\end{split}
					\end{equation}
				\end{cor}
				\normalsize
				\begin{lem}\label{x_upper_bdd}
					If ${Z_i}$, $i\in \{1,2,...,n\}$,  are i.i.d. standard normal random variables, then
					\begin{equation*}
					|Z_i| \leq n^{1/2-\epsilon}  \quad \, \forall \,\, i \qquad \text{	holds almost surely}.
					\end{equation*}
				\end{lem}
				\begin{proof} By  
					Fact \ref{normal_tail}:
					\begin{eqnarray*}
						&& P(  \text{ for all i},  |Z_i|\leq n^{1/2-\epsilon}  )\\
						&=&\bigg(1-P(|Z_1|\geq n^{1/2-\epsilon} ) \bigg)^n \\
						&=&\bigg(1-2 \frac{\frac{1}{\sqrt{2\pi}} \exp\{-\frac{1}{2}n^{1-2\epsilon} \} }{n^{1/2-\epsilon} }+O\bigg(\exp\left\{\frac{-n^{1-2\epsilon}}{2}\right\}\bigg) \bigg)^n \\
						&=& \bigg(1- \frac{\frac{2n^{1/2-\epsilon}}{\sqrt{2\pi}} \exp\{-\frac{1}{2}n^{1-2\epsilon} \} }{n}+o(n^{-2}) \bigg)^n \\
						&=& 1 - O\bigg( 2n^{1/2+\epsilon} \exp\{-\frac{1}{2}n^{1-2\epsilon} \} \bigg) \\
						&=& 1+o(1)	\,.	%
					\end{eqnarray*}
				\end{proof}

				\subsection{Adaptive Choice of $\tau^2$}

	\begin{lem}\label{lem_max_tausq_FPP}
					\begin{equation}
\label{eq.argmax}
					\begin{split}
					\argmax_{\tau^2} \bigg[ \bigg(1+\frac{1-\rho}{\tau^2} \bigg) \log \bigg(n\,
					\frac{p}{(1-p)(1-r)}
					\sqrt{\frac{1-\rho+\tau^2}{1-\rho} } \bigg)+o(1)\bigg] \\
					=(1-\rho)\bigg(2\log n + \log\log n + 2\log\frac{p}{(1-p)(1-r)}+\log2+o(1)\bigg)\,.
					\end{split}
					\end{equation}
				\end{lem}
				\begin{proof}
					Letting $ x = \frac{1-\rho}{\tau^2}$ and $c' = \frac{p}{(1-p)(1-r)}$, the expression in square
brackets in (\ref{eq.argmax}) can be written
					\begin{equation*}
									f(x) =(1+x) \big( \log (n\,c') + 1/2 \, log\, (1+ 1/x)\big) \,.
					\end{equation*}
Clearly \begin{equation*}
										f'(x) = \frac{1}{2} \big(2\log\,(n\,c') +  \,\log\,(1+1/x) -1/x\big)\,,
										\end{equation*}
					so that, $f'(x)=0$ when $1/x=2\log\,n + \log\log n +2\log c' +\log 2+o(1)$,
					or \[\tau^2= (1-\rho)(2\log\,n + \log\log n +2\log c' +\log 2)+o(1)\,.\].	
				\end{proof}

				\begin{fact} [Weak law for triangular arrays (WLTA)] \label{WLLN}
					For each $n$, let $X_{n,i}$, $1\leq k \leq n$ be independent. Let $\beta_n>0$ with $\beta_n \rightarrow \infty$ and let $\boldsymbol{\bar{x}}_{n,k}=X_{n,k}1_{\{|X_{n,k}| \leq \beta_n \}}$. Suppose that as $n\rightarrow \infty$:
					$\sum \limits^n_{k=1} P(|X_{n,k}|>\beta_n) \rightarrow 0$ and
					$1/\beta_n^2\sum \limits^n_{k=1} E\bar{X}_{n,k}^2 \rightarrow 0$\,.
					then
					\begin{center}
						\[\frac{(S_n-\alpha_n)}{\beta_n} \rightarrow 0 \text{ in probability }\]
						where $S_n=X_{n,1}+...+X_{n,n}$ and $\alpha_n=\sum \limits^n_{k=1} E\bar{X}_{n,k}$.
					\end{center}
				\end{fact}
				
				See \cite{durrett2010probability} for the proof.


				\begin{thm} \label{EB_limit_general}
					If $c_n\in(0,1) \, \forall n$ and
					$1-c_n = o(1)$, then
					\begin{eqnarray*}
						\lim \limits_{n\rightarrow \infty} \frac{1}{n}\sqrt{1-c_n} \sum \limits_{i=1}^n \exp \left\{\frac{c_n}{2}z_i^2 \right\} =
						\lim \limits_{n\rightarrow \infty}
						2\Phi\bigg(\sqrt{\frac{2(1-c_n)}{c_n} \log\, \frac{n}{\sqrt{1-c_n}} } \bigg)  -1
					\end{eqnarray*}
					in probability.
				\end{thm}
				
				\begin{proof}
					Take $X_{n,i}=exp \left\{\frac{c_n}{2}z_i^2 \right\}$; $\beta_n=\frac{n}{\sqrt{1-c_n}}$ in Fact \ref{WLLN}.
					

					\item Checking the first assumption of the WLTA:
					\begin{align*}
					P(|X_{n,i}|>\beta_n) &= P\bigg(|z_i|>\sqrt{\frac{2}{c_n}\log \frac{n}{\sqrt{1-c_n}}} \bigg) \\
					&=  2 \frac{\frac{1}{2\pi}  exp\left\{\frac{-1}{2} \frac{2}{c_n} \log \frac{n}{\sqrt{1-c_n}} \right\} }
					{\sqrt{\frac{2}{c_n} \log \frac{n}{\sqrt{1-c_n}}}}+
					O\bigg(\frac{(\frac{n}{\sqrt{1-c_n}})^{-\frac{1}{c_n}}}{(\frac{1}{c_n}log\,\frac{n}{\sqrt{1-c_n}})^3}\bigg) \\
					&= \frac{1}{\sqrt{\pi}} \frac{\sqrt{1-c_n}^\frac{1}{c_n}}{n^\frac{1}{c_n}}
					\frac{1}{\sqrt{log \frac{n}{\sqrt{1-c_n}}}} (1+o(1)) \\
					&< \frac{1}{\sqrt{\pi}} \frac{\sqrt{1-c_n}^\frac{1}{c_n}}{n^\frac{1}{c_n}} \frac{1}{\sqrt{log \,n}}(1+o(1))\,.
					\end{align*}
					Therefore,
					\begin{align}\label{eq_assump_1}
					\notag
					&\sum \limits_{i=1}^n P(|X_{n,k}|>\beta_n) =n P(|X_{n,k}|>\beta_n)\\ \notag
					&< n^{1-\frac{1}{c_n}}
					(1-c_n)^\frac{1}{2c_n}
					\frac{1}{\sqrt{log \, n }}  \\
					&=n^{-\frac{1-c_n}{c_n}}           (1-c_n)^\frac{1}{2c_n}         \frac{1}{\sqrt{log \, n }} \notag
					\rightarrow 0 \,.
					\end{align}

					\item Checking the second assumption of the WLTA: \\
Since $\lim \limits_{n\rightarrow \infty} c_n \rightarrow 1$, without loss of generality, assume $c_n>3/4$. Then
	\small{
					\begin{eqnarray*}
						&&\hspace{-0.5cm} \frac{1}{\beta_n^2} \sum \limits^n_{k=1} E{\bar{X}}_{n,k}^2  =
						\frac{1-c_n}{n^2} n \int \limits_{|z|<\sqrt{\frac{2}{c_n}log \frac{n}{\sqrt{1-c_n}}}}
						\exp \left\{c_n z^2 \right\}
						\frac{1}{\sqrt{2\pi}} \exp \left\{\frac{-1}{2}z^2 \right\} dz \\
						&=& \frac{1-c_n}{n}\frac{1}{\sqrt{2\pi}}  \left\{ \hspace{-2.5cm}
						\int \limits_{\phantom{mmmmmmmmmmm}1<|z|<\sqrt{\frac{2}{c_n}log \frac{n}{\sqrt{1-c_n}}}}
						\hspace{-2.5cm} \exp \left\{(c_n-\frac{1}{2}) z^2 \right\}   dz +
						\int \limits_{|z|<1}
						\exp \left\{(c_n-\frac{1}{2}) z^2 \right\}   dz \right\}   \\
						&\leq & \frac{1-c_n}{n} \frac{1}{\sqrt{2\pi}} \left\{
						\int \limits_{1<|z|<\sqrt{\frac{2}{c_n}log \frac{n}{\sqrt{1-c_n}}}}
						z\exp \left\{(c_n-\frac{1}{2}) z^2 \right\}   dz +d \right\}\\
						&=&   \frac{1-c_n}{n} \frac{1}{\sqrt{2\pi}} \left\{  \frac{1}{2(c_n-\frac{1}{2})}
						\exp \left\{ (c_n-\frac{1}{2})(\frac{2}{c_n})\, \log \frac{n}{\sqrt{1-c_n}}  \right\}+d'\right\}\\
						&=&\frac{1}{\sqrt{2\pi}}
						\bigg( \frac{1}{2c_n-1} \bigg) \frac{1-c_n}{n}
						\bigg( \frac{n}{\sqrt{1-c_n}}\bigg) ^{2-\frac{1}{c_n}} +o(1)\\
						&=& \frac{1}{\sqrt{2\pi}}
						\underbrace{\bigg( \frac{1}{2c_n-1} \bigg) }_\textrm{$\leq 2 $}
						n^{1-\frac{1}{c_n}}
						(1-c_n)^\frac{1}{2c_n}  +o(1) \\
						&\leq&  \frac{2}{\sqrt{2\pi}}
						n^{-\frac{1-c_n}{c_n}}
						(1-c_n)^\frac{1}{2c_n}  +o(1) 
		                =o(1)\,.
					\end{eqnarray*}
		}
					Noting that
					\begin{eqnarray*}
						&& \frac{\sqrt{1-c_n}}{n} \alpha_n = \frac{1-c_n}{n} \sum \limits^n_{i=1} E\bar{X}_{n,i} \\
						&=& (1-c_n) \int \limits_{|z|<\sqrt{\frac{2}{c_n}log \frac{n}{\sqrt{1-c_n}}}}
						e^{\frac{c_n z^2}{2}} \frac{1}{\sqrt{2\pi}} e^{\frac{-z^2}{2}} dz 
					 = 2 \bigg(\Phi \bigg(\frac{ \sqrt{\frac{2}{c_n}\log \frac{n}{\sqrt{1-c_n}}}}
						{\sqrt{(1-c_n)^{-1}}} \bigg) -\frac{1}{2} \bigg)\,,
					\end{eqnarray*}
					the WLTA yields
					\begin{eqnarray*}
						\frac{S_n-\alpha_n}{\beta_n} &=& \frac{\sum \limits_{i=1}^n e^{c_n z_i^2} -\alpha_n}
						{\frac{n}{\sqrt{1-c_n}}}  
						= \frac{ \sqrt{1-c_n} \sum \limits_{i=1}^n e^{c_n z_i^2}}{n} -\frac{\sqrt{1-c_n}}{n}\alpha_n \rightarrow 0\,.
					\end{eqnarray*} in probability, and the result follows.
				\end{proof}

				\begin{cor} \label{EB_limit}
					Letting $c_n=\frac{\hat{\tau}^2_n }{1-\rho+\hat{\tau}^2_n }$,
					\begin{eqnarray*}
						\frac{1}{n} \sqrt{1-c_n} \sum \limits_{i=1}^n \exp \left\{\frac{c_n}{2}z_i^2 \right\} \rightarrow
						\begin{cases}
							1 &\mbox{if }  \frac{log\,n }{\hat{\tau}^2_n }\rightarrow \infty\,,\\
							2\Phi\bigg(\sqrt{\frac{2}{k}}\bigg)-1 & \mbox{if } \frac{log\,n }{\hat{\tau}^2_n }\rightarrow  \frac{1}{(1-\rho)k}\,, \\\
							0 & \mbox{if } \frac{log\,n }{\hat{\tau}^2_n }\rightarrow 0\,
						\end{cases}
					\end{eqnarray*}
					in probability.
				\end{cor}
				
				\begin{proof}
					By Theorem \ref{EB_limit_general}:
					\item \textit{Case I}: $\frac{log\,n}{\hat{\tau}^2_n }\rightarrow \infty $. Clearly
					\[ \frac{\sqrt{1-c_n}}{n} \alpha_n\Rightarrow 2 \Phi \bigg(\sqrt{\frac{2(1-\rho)}{\hat{\tau}^2_n }log\bigg( n\sqrt{\frac{1-\rho+\hat{\tau}^2_n }{1-\rho}}\bigg)} \bigg)-1 \rightarrow 1\,.\]
					
					\item \textit{Case II}: $\frac{log\,n}{\hat{\tau}^2_n }\rightarrow \frac{1}{(1-\rho)k} $. Clearly
					\[\frac{\sqrt{1-c_n}}{n} \alpha_n\Rightarrow 2 \Phi \bigg(\sqrt{\frac{2(1-\rho)}{\hat{\tau}^2_n }log\bigg( n\sqrt{\frac{1-\rho+\hat{\tau}^2_n }{1-\rho}}\bigg)} \bigg)-1 \rightarrow 2\Phi\bigg(\sqrt{\frac{2}{k}}\bigg)-1\,. \]
					
					\item \textit{Case III}: $\frac{log\,n}{\hat{\tau}^2_n }\rightarrow 0 $. Clearly
					\[\frac{\sqrt{1-c_n}}{n} \alpha_n\Rightarrow 2 \Phi \bigg(\sqrt{\frac{2(1-\rho)}{\hat{\tau}^2_n }log\bigg( n\sqrt{\frac{1-\rho+\hat{\tau}^2_n }{1-\rho}}\bigg)} \bigg)-1 \rightarrow 0\,.\]
				\end{proof}

				\begin{lem} 
					\begin{equation} \label{EB_claim}
					\begin{split}	
					\lim\limits_{n\rightarrow \infty}&\frac{1}{n\sqrt{1+\tau^2_n a}} \sum \limits_{i=1}^n \exp \left\{\frac{\tau^2_n}{2(1+\tau^2_n a)} \big( \frac{x_i}{1-\rho}+bn\boldsymbol{\bar{x}} \big)^2 \right\} \\
					=&\frac{1}{n}
					\sqrt{\frac{1-\rho}{1-\rho+\tau^2_n }}
					\sum \limits_{i=1}^n \exp \left\{\frac{\tau^2_n z_i^2 }{2(1-\rho+\tau^2_n )} \right\}(1+o(1)) \quad a.s.
					\end{split}
					\end{equation}
				\end{lem}
				
				\begin{proof}  Expanding the coefficients yields
					\begin{align*}
					\frac{1}{1+\tau^2_n a} &= \bigg(1+\frac{\tau^2_n (1+(n-2)\rho)}{(1+(n-1)\rho)(1-\rho)}\bigg)^{-1} \\
					&=\frac{1-\rho}{  1-\rho +\tau^2_n  \big( 1+\frac{-\rho}{1+(n-1)\rho}\big) }
					=\frac{1-\rho} {1-\rho+\tau^2_n  }(1+O(1/n))\,,
					\end{align*}
					and \small
					\begin{align*}
					&\bigg( \frac{x_i}{1-\rho}+bn\boldsymbol{\bar{x}} \bigg)^2
					= \frac{1}{(1-\rho)^2} \bigg( x_i +\frac{-\rho n \boldsymbol{\bar{x}}}{1+(n-1)\rho}\bigg)^2 \\
					&=\frac{1}{(1-\rho)^2} \bigg(x_i-\boldsymbol{\bar{x}}\big( 1- \frac{1-\rho}{1-\rho+\rho n}\big) \bigg)^2 \\
					&=\frac{1}{(1-\rho)^2}  \bigg(\sqrt{1-\rho} z_i+
					\sqrt{\rho} z 
					\underbrace{ \big( \frac{1-\rho}{1-\rho+\rho n} \big)}_\textrm{$O(1/n)$}+
					\sqrt{1-\rho}
					\hspace{-0.3cm}
					\underbrace{\boldsymbol{\bar{z}}}_\textrm{$O(1/\sqrt{n})$} 
					\hspace{-0.3cm}
					\big(-1+\frac{1-\rho}{1-\rho+\rho n} \big) \bigg)^2\\
					&=\frac{z_i^2}{1-\rho}+O\big((\log\,n)/\sqrt{n}) \big)\,.
					\end{align*}
					\normalsize
					
					\item Therefore,
					\begin{align*}
					&  \frac{1}{\sqrt{1+\tau^2_n  a}} 1/n
					\sum \limits_i \exp \left\{\frac{\tau^2_n }{2(1+\tau^2_n  a)}
					\big( \frac{x_i}{1-\rho}+bn\boldsymbol{\bar{x}} \big)^2 \right\}
					\\
					&=
					\sqrt{\frac{1-\rho}{1-\rho+\tau^2_n +o(1)}} 1/n
					\sum \limits_i \exp \left\{\frac{\tau^2_n }{2} \bigg[
					\frac{z_i^2}{1-\rho+\tau^2_n }+o(1)
					\bigg] \right\}\,.
					\end{align*}
				\end{proof}

	\begin{lem}\label{EB_max}
		Under the null model, suppose \[\max \limits_j \bigg(\frac{x_j-\boldsymbol{\bar{x}}}{\sqrt{1-\rho}} \bigg)^2=2\log(n)+ \log\log(n)+c\,.\]
		Then
\[ L_n(\tau^2) = \frac{1}{n}
					\sqrt{\frac{1-\rho}{1-\rho+\tau^2 }}
					\sum \limits_{i=1}^n \exp \left\{\frac{\tau^2 z_i^2 }{2(1-\rho+\tau^2 )} \right\} \]
 is maximized at
		\[ {\hat \tau}^2_n =(1-\rho) k(c) (\log\,n) (1+o(1)) \,,\]
		where
		\[k(c)=(1+2/\sqrt{\pi}
		\exp \left\{-c/2 \right\} )^{-1}\,.\]
		
	\end{lem}	
		\begin{proof}  
			Without loss of generality, let $\max \limits |z_i|=|z_1|$.
			
			\begin{align*}
			 L_n(\hat{\tau}^2_n )
			&= \Biggl( \,\,
			\underbrace{\sqrt{\frac{1-\rho}{1-\rho+\hat{\tau}^2_n }}1/n\exp\left\{ \frac{\hat{\tau}^2_n  z_1^2 }{2(1-\rho+\hat{\tau}^2_n ) } 
				\bigg) \right\} }_{I}+ \\
			& \qquad  \underbrace{\sqrt{\frac{1-\rho}{1-\rho+\hat{\tau}^2_n }}1/n \sum\limits_{i=2}^n \exp\left\{
				\frac{\hat{\tau}^2_n z_i^2}{2(1-\rho+\hat{\tau}^2_n )}
				\right\}\Biggr)}_{II}(1+o(1))\,.
			\end{align*}
			\normalsize
			First, note that $L_n(\hat{\tau}^2_n )\rightarrow 0$ when $\log n / {\hat \tau}^2_n \rightarrow \infty$, since 					
		\begin{equation*}
			\begin{split}
			I &=
		 \frac{1}{\sqrt{\hat{\tau}^2_n} }
		 n^{\frac{-(1-\rho)}{1-\rho+\hat{\tau}^2_n}}
		 \sqrt{\log n}^{\frac{(\hat{\tau}^2_n)}{1-\rho+\hat{\tau}^2_n}}
			 	e^{\frac{c\hat{\tau}^2_n }{2(1-\rho+\hat{\tau}^2_n )}} 	(1+o(1))\\
		 &= n^{\frac{-(1-\rho)}{1-\rho+\hat{\tau}^2_n}}
		 \sqrt{\log n /\hat{\tau}^2_n} e^{c/2} (1+o(1))
		 =o(1) \,,\\
			 	%
		II& \rightarrow 0  \text{ by Corollary }\ref{EB_limit}.
			\end{split}
			\end{equation*}
	 Similarly, one can show that $L_n(\hat{\tau}^2_n )\rightarrow 1 $ when $\log n/\hat{\tau}^2_n  \rightarrow 0$, since
			\begin{equation*}
			\begin{split}
			\qquad I &=  n^{\frac{-(1-\rho)}{1-\rho+\hat{\tau}^2_n}}
			\sqrt{\log n /\hat{\tau}^2_n} e^{c/2} (1+o(1))
					=o(1) \\
			II& \rightarrow 1  \text{ by Corollary }\ref{EB_limit}.
			\end{split}
			\end{equation*}
	
	For the case in which $\log n / {\hat \tau}^2_n \rightarrow k$, using Corollary \ref{EB_limit}, it follows that
	\[	L_n({\hat \tau}^2_n) = [v e^{(\frac{c}{2} - v^2)} + 2 \Phi( \sqrt{2} v) -1](1+o(1)), \]
	where $v = \sqrt{(1-\rho)/k}$. Differentiating $f(v) = [v e^{(\frac{c}{2} - v^2)} + 2 \Phi( \sqrt{2} v)]$ and setting the derivative to 0, yields the solution $\hat{v}= \sqrt{\frac{1}{2} - \frac{1}{\sqrt{\pi}}e^{-c/2}}$, which translates into $k(c)$ as in the statement of the
	lemma. It is straightforward to show that this extrema of $f(v)$ is the maximum, and 
	\[f(\hat{v}) >  \max \{\lim_{v\rightarrow 0}  f(v), \lim_{v\rightarrow \infty} f(v) \} = 1\,.\]
As this maximum thus exceeds the maximum over the domains $\log n / {\hat \tau}^2_n \rightarrow \infty$ and $\log n/\hat{\tau}^2_n  \rightarrow 0$, the proof is complete.
\end{proof}
	
			
	\begin{lem}\label{l>0}
		For the $k(c)$ defined above,
		\begin{equation*}
		\log(k(c)/2) + 2/k(c) -1 > 0 \,\forall \,c > 0
		\end{equation*}
	\end{lem}
	\begin{proof} Note that $x = k/2 <1$, so that we want to show that  $f(x) = \log(x) + 1/x - 1 >0$ over this region.
Since $f'(x)= 1/x -1/x^2 <0 $ over this region, $f(x)$ is minimized at $x=1$, proving the result.
		
	\end{proof}				
				

				
				
				
				
				
				
			\end{document}